\documentclass{amsart}

\usepackage{graphicx}
\usepackage{subcaption}

\newtheorem{theorem}{Theorem}[section]
\newtheorem{corollary}[theorem]{Corollary}
\newtheorem{lemma}[theorem]{Lemma}

\theoremstyle{definition}
\newtheorem{definition}[theorem]{Definition}

\theoremstyle{remark}

\numberwithin{equation}{section}

\begin{document}
\title{DNA Origami and Unknotted A-trails in Torus Graphs}

\author[Morse]{Ada Morse}

\address{Dept. of Mathematics and Statistics, University of Vermont, Burlington, VT 05405}

\author[Adkisson]{William Adkisson}

\address{Dept. of Mathematics, University of Chicago, Chicago IL, 60637}

\author[Greene, Perry, Smith, Ellis-Monaghan]{Jessica Greene, David Perry, Brenna Smith,  Jo Ellis-Monaghan}

\address{Dept. of Mathematics, Saint Michael’s College, Colchester VT, 05439}

\author[Pangborn]{Greta Pangborn}

\address{Dept. of Computer Science, Saint Michael’s College, Colchester VT, 05439}

\begin{abstract}
Motivated by the problem of determining unknotted routes for the scaffolding strand in DNA origami self-assembly, we examine existence and knottedness of A-trails in graphs embedded on the torus. We show that any A-trail in a checkerboard-colorable torus graph is unknotted and characterize the existence of A-trails in checkerboard-colorable torus graphs in terms of pairs of quasitrees in associated embeddings. Surface meshes are frequent targets for DNA nanostructure self-assembly, and so we study both triangular and rectangular torus grids. We show that, aside from one exceptional family, a triangular torus grid contains an A-trail if and only if it has an odd number of vertices, and that such an A-trail is necessarily unknotted. On the other hand, while every rectangular torus grid contains an unknotted A-trail, we also show that any torus knot can be realized as an A-trail in some rectangular grid. Lastly, we use a gluing operation to construct infinite families of triangular and rectangular grids containing unknotted A-trails on surfaces of arbitrary genus. We also give infinite families of triangular grids containing no unknotted A-trail on surfaces of arbitrary nonzero genus.
\end{abstract}

\keywords{Eulerian circuits, A-trails, DNA, Knots, Knotted graphs}

\subjclass[2000]{Primary 05C10, 57M25, 05C45, 05C62}

\maketitle

\section{Introduction}

We study a new application of knot theory arising from the context of DNA nanostructure self-assembly. There has been increasing interest recently in the DNA origami method \cite{Rot05, Rot06} for self-assembly of structures that can be modeled by embedded graphs, i.e. polyhedral skeletons and spherical triangulations \cite{CS91}  \cite{AD+09}  \cite{GCXS09} \cite{TVNYG11}  \cite{IKJSWY14} \cite{DNA15}. The origami method uses a single circular (unknotted) strand of DNA called the \emph{scaffolding strand}. This strand is paired with a collection of short \emph{staple strands}. Each staple strand has bases chosen complementary to certain intervals along the scaffolding strand. By bonding to these intervals, the staple strands fold the scaffolding strand into the target structure.

Given a particular target structure, designs using the origami method require the identification of a route through the structure for the scaffolding strand to follow. Since the scaffolding strand is an unknot, these routes must be unknotted. Herein lies the connection between knot theory and DNA self-assembly. 

Because of the way DNA strands stack, it is preferable that neither scaffolding nor staple-strands interweave, that is, cross over one another. Thus, when the target structure is modeled as a graph embedded on a surface in $3$-space, optimal routes correspond to \emph{A-trails} -- Eulerian circuits that turn either left or right at each vertex (see e.g. \cite{DNA15}). See Figure \ref{fig:atrailex} for an example of this design proces, and note that given an A-trail, there is a natural placement of staple strands avoiding interweaving. If the surface is a sphere, as in \cite{DNA15}, all A-trails in the target structure are necessarily unknotted, and so may be used as routes for the unknotted scaffolding strand. On higher-genus surfaces such as the torus, this is no longer the case \cite{EMPSBDFHMSW}.

\begin{figure}
    \centering
     \begin{subfigure}[b]{0.5\columnwidth}
     \centering
    \includegraphics[scale=.2]{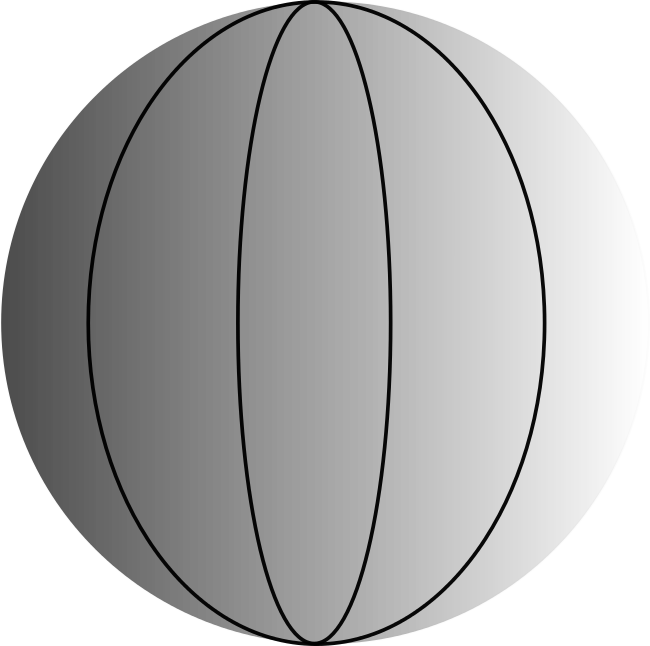}
    \caption{The target structure on a sphere.}
    \label{subfig:target}
    \end{subfigure}%
    \begin{subfigure}[b]{0.5\columnwidth}
    \centering
    \includegraphics[scale=.2]{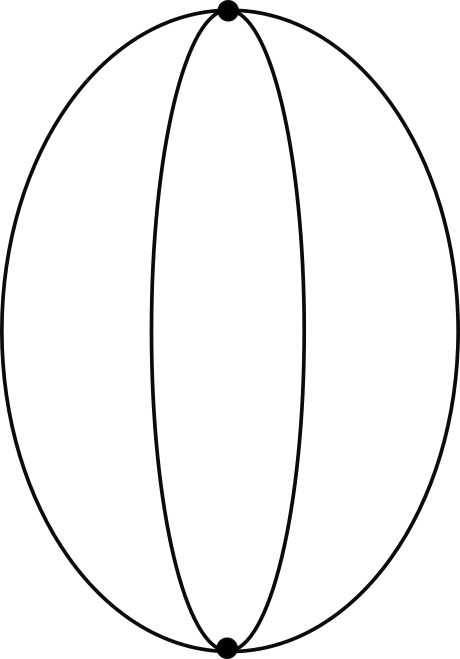}
    \caption{Modeling the target as a (plane) graph.}
    \label{subfig:target}
    \end{subfigure}
    \begin{subfigure}[b]{0.5\columnwidth}
    \centering
    \includegraphics[scale=.2]{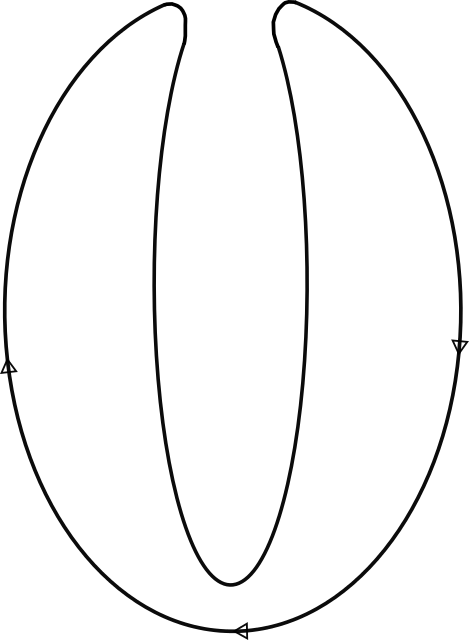}
    \caption{An A-trail in the model.}
    \label{subfig:atrail}
    \end{subfigure}%
    \begin{subfigure}[b]{0.5\columnwidth}
    \centering
    \includegraphics[scale=.2]{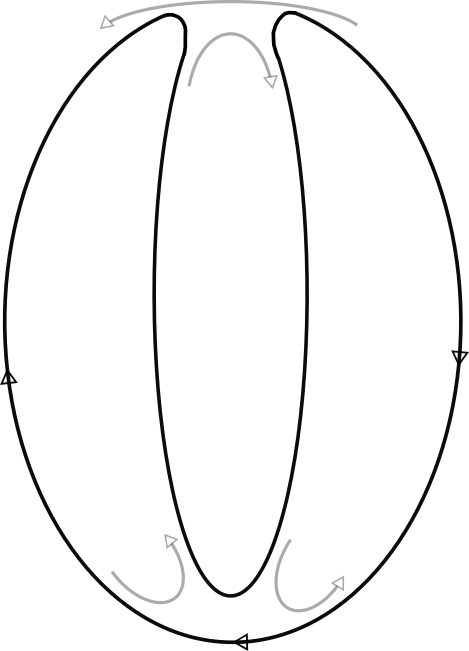}
    \caption{Placing the staple strands.}
    \label{subfig:target}
    \end{subfigure}
    \caption{The origami design process using A-trails.}
    \label{fig:atrailex}
\end{figure}

The existence of A-trails has been studied almost exclusively in the case of plane graphs, in particular plane 4-regular graphs. Determining if a particular plane graph has an A-trail is NP-hard in general, however, in \cite{K68} Kotzig showed that every 4-regular plane graph has an A-trail, and sufficient conditions were discovered for the existence of A-trails in $2$-connected plane graphs in \cite{ANDERSEN199899}. In addition to the plane, Andersen et al have characterized existence of orthogonal pairs of  A-trails in checkerboard-colorable 4-regular graphs on the torus and projective plane \cite{ANDERSEN1996232}. The connection between knotted A-trails and DNA origami self-assembly design was originally discussed, and the existence of knotted A-trails demonstrated, in \cite{EMPSBDFHMSW}.

In this paper, we study the existence and knottedness of A-trails in graphs embedded on the torus. We begin by studying checkerboard-colorable embeddings, i.e. embeddings whose faces can be properly $2$-colored, as these colorings have often been exploited in the past to study A-trails both on the plane (where every Eulerian embedding is  checkerboard-colorable) \cite{ANDERSEN199899} and on the torus (where not every Eulerian embedding is checkerboard-colorable) \cite{ANDERSEN1996232}. We show that any A-trail in a checkerboard-colorable toroidal graph is unknotted. However, we also exhibit an infinite family of checkerboard-colorable toroidal graphs containing no A-trails. In the process of proving this result, we construct a bijection between A-trails in a checkerboard-colorable toroidal graph and certain pairs of quasitrees in associated embeddings. We note that Andersen and Bouchet have proven a similar result for orthogonal pairs of A-trails in the $4$-regular case \cite{ANDERSEN1996232}.

Surface meshes are of particular interest for the application in nanostructure self-assembly \cite{DNA15}. On the torus, these are the $6$-regular triangular grids and the $4$-regular rectangular grids. We begin by analyzing triangular torus grids. Using Altshuler's construction of all regular torus triangulations \cite{ALTSHULER1973201}, we show that triangular torus grids are checkerboard-colorable. Hence, using results from the first section, we prove that a triangular torus grid with a non-Hamiltonian straight-ahead walk contains an A-trail if and only if it has an odd number of vertices, and in that case every A-trail is unknotted. This result yields an infinite family of triangular torus grids with unknotted A-trails as well as an infinite family of triangular torus grids with no A-trails (in particular, without unknotted A-trails.)

We also analyze rectangular torus grids. Unlike the triangular grids, these may contain knotted A-trails. Indeed, we provide an example of a rectangular torus grid containing precisely two A-trails, each of which is knotted. Restricting to a naturally-defined family of rectangular torus grids, we show that each grid contains unknotted A-trails as well as, in many cases, knotted A-trails and non-trivially linked circuit decompositions. As a consequence, we show that every torus knot (indeed, link) can be constructed from an A-trail (smooth circuit decomposition) of some rectangular grid.

While our main results are specific to the torus, we close by considering composites of triangular and rectangular torus grids on higher-genus surfaces. To do so, we define and study a connected sum for circuit decompositions, which allows certain results on the torus to be lifted to the $n$-torus. In particular, we construct infinite families of composite triangular and rectangular grids containing unknotted A-trails on any $n$-torus. We close by briefly discussing the types of knots obtainable from composites of rectangular grids.

While this paper was motivated by practical application, we believe the study of links formed from circuit decompositions of embedded graphs is of independent mathematical interest as it provides a new perspective on knotted and linked structures in (Eulerian) embedded graphs.

\section{Definitions}

\subsection{Graph theory.} We first recall some standard terminology and definitions from graph theory. A \emph{graph} $G$ consists of a set $V(G)$ of vertices together with a set $E(G)$ of edges connecting vertices.  We allow multiple edges between the same two vertices, as well as edges from a vertex to itself (\emph{loops}.) Graphs with neither multiple edges nor loops are called \emph{simple}. Two vertices are \emph{adjacent} if there is an edge between them, in which case they are that edge's \emph{endpoints}. We think of each edge having two half-edges, one each incident to its endpoints. The number of half-edges incident to a vertex is called the \emph{degree} of $v$ and is denoted $\deg_G(v)$ (note that both half-edges of a loop are counted). The graph is $k$-regular if the degree of each vertex is $k$. A \emph{walk} in $G$ is a sequence $v_1e_1v_2e_2 \cdots v_k e_k v_{k+1}$ of vertices $v_i$ and edges $e_i$ such that for each $i$ the vertices $v_i$ and $v_{i+1}$ are the endpoints of the edge $e_i$.  If $v_{k+1} = v_1$, the walk is \emph{closed}. A graph is \emph{connected} if for each pair $v$ and $w$ of vertices there is a walk beginning at $v$ and ending at $w$, and a \emph{connected component} of a graph is a maximal connected subgraph. A \emph{circuit} is a closed walk repeating no edges, and a \emph{cycle} is a closed walk repeating no vertices (and hence also repeating no edges.) The \emph{length} of a cycle is the number of edges it contains. The \emph{n-cycle} $C_n$ is the graph consisting only of a cycle on $n$ vertices. An \emph{Eulerian circuit} is a circuit containing every edge of $G$, and a graph with an Eulerian circuit is \emph{Eulerian}. Recall the standard result that a graph is Eulerian if and only if it contains a most one nontrivial connected component and every vertex has even degree. A \emph{circuit decomposition} of an Eulerian graph is a collection of edge-disjoint circuits $C$ such that every edge is contained in an element of $C$.

Since the targets of DNA nanostructure self-assembly are typically geometric structures in $\mathbb{R}^3$, we will be concerned with graphs embedded in $\mathbb{R}^3$. In particular, following the constructions of triangulations of a sphere using DNA in \cite{DNA15}, we will be interested in grid-like graphs on orientable surfaces. A \emph{cellular embedding} of a graph $G$ on a surface $\Sigma$ is a drawing $\Gamma$ of $G$ on $\Sigma$ such that edges meet only at vertices, and $\Sigma \setminus \Gamma$ is the disjoint union of open disks, called \emph{faces}. Given a face $f$, we denote by $\partial f$ its topological boundary. Note that $\partial f$ always consists of a closed walk in $G$. When this walk is a cycle, we will use $\partial f$ to refer also to this cycle as a graph. The \emph{geometric dual} of $G$ is the cellularly embedded graph $G^*$ obtained by placing a vertex in each face of the embedding of $G$, and drawing edges connecting these vertices, when their associated faces share an edge, through that edge. A \emph{thickened graph} is obtained from $G$ by thickening vertices to disks and edges to ribbons (so $G$ may be obtained by contracting the vertices of the thickened graph to points and the edges of the thickened graph to lines.) A graph is a \emph{quasitree} if it has a single boundary component, viewed as a thickened graph in $\Sigma$. We refer to reader to \cite{Tucker} for details on thickened graphs. Lastly, recall that a cellular embedding of $G$ on a surface $\Sigma$ induces, at each vertex, a \emph{cyclic rotation of half-edges} obtained by listing the half-edges incident to $v$ cyclically in counter-clockwise order. 

\subsection{Knot theory.} We also establish some standard knot theoretical definitions, following conventions of \cite{BHM2014,Lickorish}. A \emph{knot} is a simple closed curve in $\mathbb{R}^3$, and a \emph{link} is a collection of pairwise disjoint knots, called its \emph{components}. We consider links to be equivalent under isotopy. The unknot, or trivial knot, is the knot having a plane projection containing no crossings. A simple closed curve is \emph{knotted} if it is a nontrivial knot and \emph{unknotted} otherwise. A link $L$ is \emph{trivial} if for each component $K$ there is a solid sphere $S$ so that $S \cap L = K$. Knots are \emph{unlinked} if they form a trivial link and \emph{linked} otherwise. Let $K$ be an oriented knot whose intersection with a plane $E$ consists of two points $p$ and $q$. The arc of $K$ from $p$ to $q$ is closed by an arc in $E$ to obtain a knot $K_1$ and the arc of $K$ from $q$ to $p$ is likewise closed to form a knot $K_2$. Then $K$ is the \emph{connected sum} of $K_1$ and $K_2$, and we write $K = K_1 \# K_2$. A knot $K$ that is the connected sum of two nontrivial knots is a \emph{composite knot}, and otherwise is \emph{prime}. Given two unlinked knots $K_1$ and $K_2$, we can always form their connected sum, and the result is independent of the choice of plane and representatives. It can be shown that the unknot is the additive identity for the connected sum of knots, i.e. if $K_0$ is the unknot, then $K_0 \# K = K$ and $K \# K_0 = K$ for all knots $K$. 

\subsection{Assembly, transitions, and A-trails.} Constructing a graph $G$ embedded on a surface from DNA using the origami method requires the identification of a feasible route for the scaffolding strand through the structure. In particular, one must identify an Eulerian circuit in the graph subject to certain turning restrictions: at each vertex in the surface, no strand of the DNA can cross over another. We use transition systems and A-trails to model this constraint. Begin by fixing an Eulerian graph $G$ embedded on a surface $\Sigma$.

\begin{definition} (transition, smoothing)
 Let $v$ be a vertex of $G$. A \textit{transition} at $v$ is a partition of the half-edges incident to $v$ into pairs. A transition $T$ at $v$ is \emph{smooth} if $T$ only pairs half-edges adjacent in the cyclic order at $v$ given by the embedding of $G$.
The embedded graph $G'$ created by performing a \textit{smoothing} on $G$ at $v$ according to a smooth transition $T$ is identical to $G$ except in a neighborhood of $v$ containing only $v$ and (portions of) its incident half-edges. In this neighborhood we delete $v$ and connect the half-edges paired by $T$ in the manner shown by Figure \ref{fig:smoothing}.
\end{definition}

Any route for the scaffold strand at a vertex must be topologically equivalent to the curves obtained by smoothing at that vertex according to a smooth transition. To construct the graph, we will need such a choice of transition at all vertices.

\begin{figure}[ht]
    \centering
    \includegraphics[scale=.1]{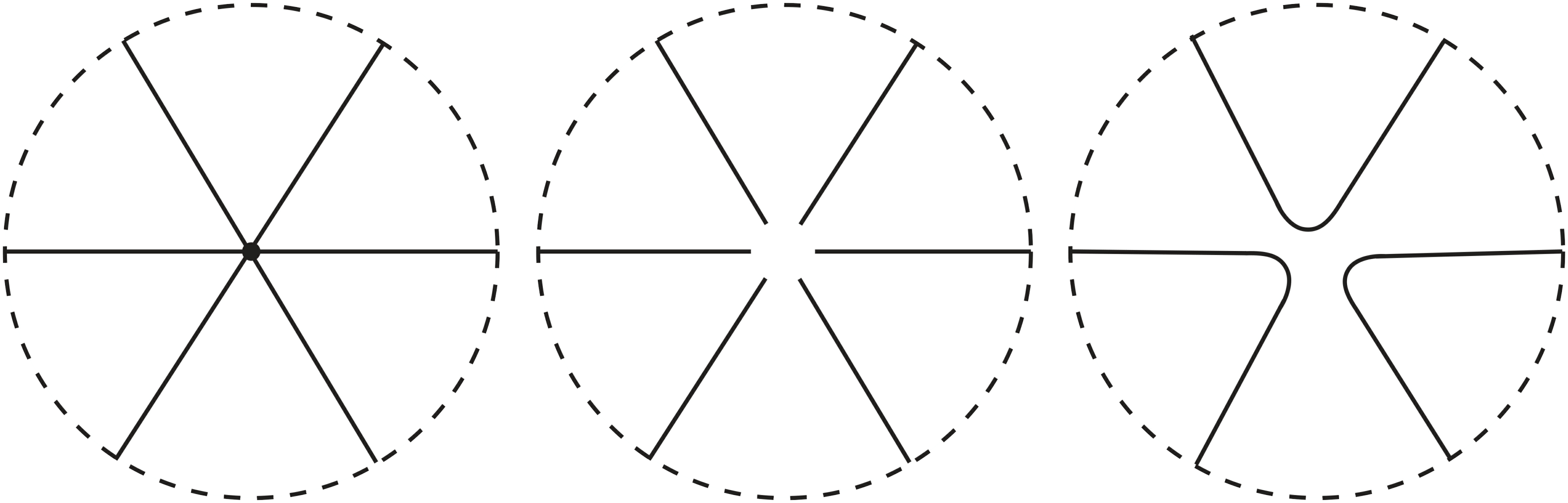}
    \caption{Performing a smoothing locally at a vertex.}
    \label{fig:smoothing}
\end{figure}

\begin{definition} (transition system)
A \textit{transition system} of $G$ is a choice of a transition at every vertex of $G$. A \textit{smooth transition system} is a transition system such that every transition is smooth.
\end{definition}

Any circuit decomposition $C$ of $G$ induces a transition system of $G$ by pairing half-edges traversed consecutively in a circuit of $C$. Likewise, given a transition system $T$, we can recover the circuit decomposition that induces it.

\begin{definition} (A-trail)
A circuit decomposition of $G$ is \emph{smooth} if it induces a smooth transition system. In particular, an Eulerian circuit that induces a smooth transition system is called an \textit{A-trail.} If $G$ contains an A-trail, we say $G$ is \emph{smoothly Eulerian.}
\end{definition}

We are interested primarily in the knot-theoretical properties of A-trails after smoothing, which the following language makes precise.

\begin{definition}
Let $C$ be a smooth circuit decomposition of $G$ inducing the transition system $T$. By smoothing $G$ at each vertex according to $T$, one obtains a link (on $\Sigma$) in $\mathbb{R}^3$, denoted by $\mathcal{L}(C)$.\end{definition}

Any knot-theoretical language applied to $C$ is to be understood in terms of $\mathcal{L}(C)$. For example, we will say an A-trail $A$ is \emph{knotted} if $\mathcal{L}(A)$ is knotted and so forth.

Our main question is the following: when is an embedded graph \emph{origami constructible}, i.e. when does an embedded graph contain an unknotted A-trail to use as the scaffolding route in the DNA origami method.

\subsection{Torus knots.} Since (most of) the graphs we consider are on the torus, we will also need some standard definitions and results from the theory of torus knots. As usual, we will be representing torus graphs and torus knots using a square with its opposite sides identified. Unless otherwise specified, we assume that a point on the top and a point on the bottom are identified if they lie on the same vertical line, and a point on the left is identified with a point on the right if they lie on the same horizontal line. A circle resulting from identifying the endpoints of a vertical line is a \emph{canonical meridian} while a circle resulting from identifying the endpoints of a horizontal line is a \emph{canonical longitude} of the torus. A \emph{meridional curve} is a simple closed curve homotopic to a canonical meridian, while a \emph{longitudinal curve} is a simple closed curve homotopic to a canonical longitude. Unless otherwise specified, we orient longitudinal curves so they go bottom to top and consider this the \emph{positive longitudinal direction} and orient meridional curves left-to-right and consider this the \emph{positive meridional direction}. We will assume in what follows that if a graph embedded on a torus in $\mathbb{R}^3$ is represented in the square, then the canonical longitude and meridion of the square are longitudinal and meridional curves, respectively, of the embedded torus in $\mathbb{R}^3$.

A \textit{torus knot} is a knot embedded onto the surface of an unknotted torus in $\mathbb{R}^3$ (see Figure \ref{fig:torus_trefoil}).  A \textit{torus link} is a link which is embedded on the surface of an unknotted torus in $\mathbb{R}^3$.

\begin{figure}[ht!]
    \centering
    \includegraphics[scale=.2]{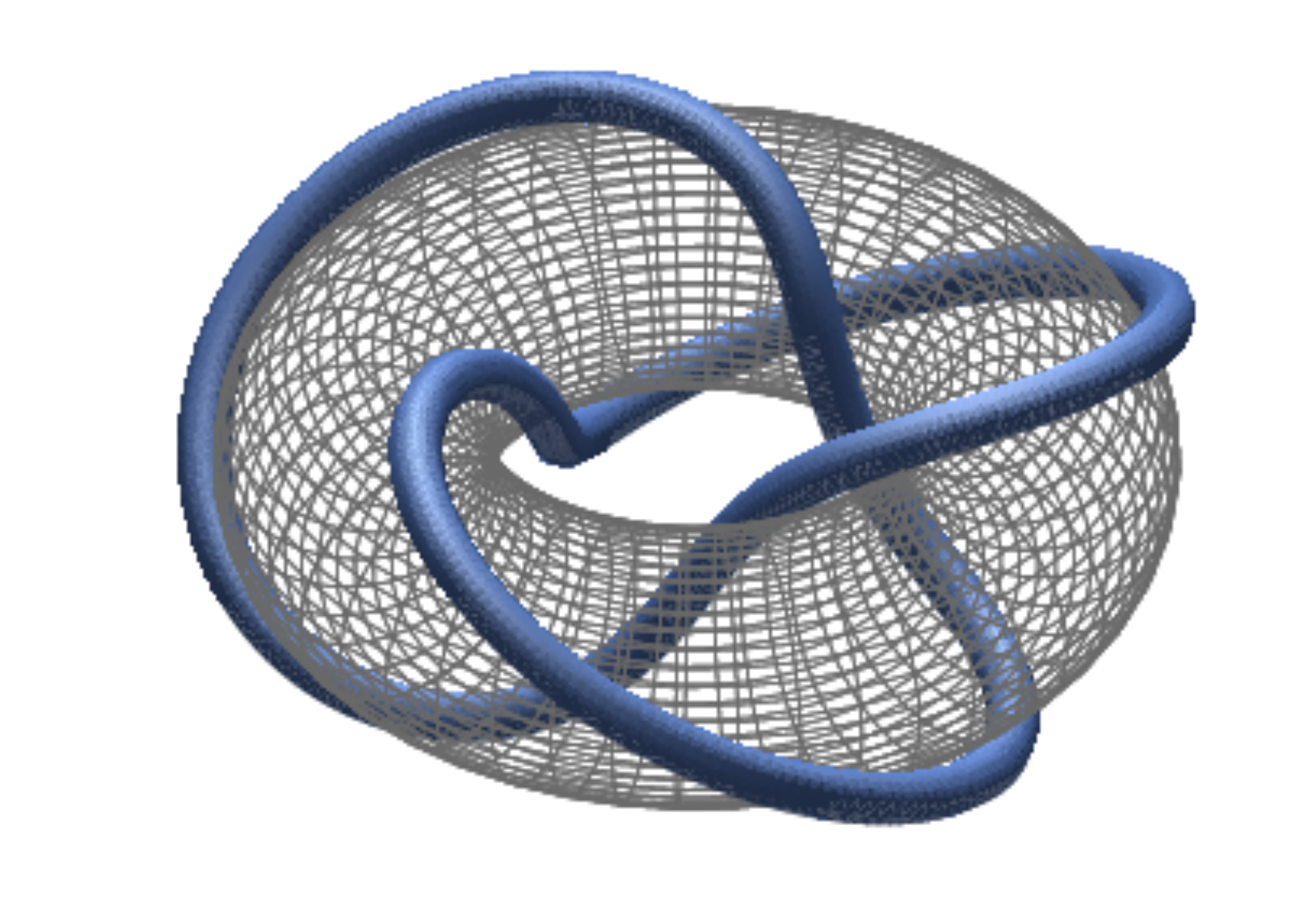}
    \caption{A trefoil knot embedded on a torus. }
    \label{fig:torus_trefoil}
\end{figure}

Recall that an oriented torus link is characterized, as a link, by the number of oriented longitudinal and meridional rotations it completes (we require that the components of a link be consistently oriented). For example, the trefoil knot in Figure \ref{fig:torus_trefoil} completes two positive meridional and three positive longitudinal rotations. Thus we say a link is \emph{the $(p,q)$ torus link} if it completes $p$ longitudinal rotations and $q$ meridional rotations, where the signs of $p$ and $q$ indicate whether these rotations are completed in the positive or negative directions. A $(p,q)$ torus link is a knot if and only if $p$ and $q$ are relatively prime. The $(p,q)$ torus knot is the unknot if and only if either $p$ or $q$ is $1$ (or both are $0$). The \emph{standard embedding} of the $(p,q)$ torus link is obtained by marking $p$ points on the left (and hence also right) of the square, $q$ points on the top (and bottom), and connecting these with lines as in Figure \ref{fig:torus_link}. We refer the reader to e.g. \cite{Lickorish} for a more rigorous and detailed exposition of torus knots.

\begin{figure}[ht!]
    \centering
    \includegraphics[scale=.2]{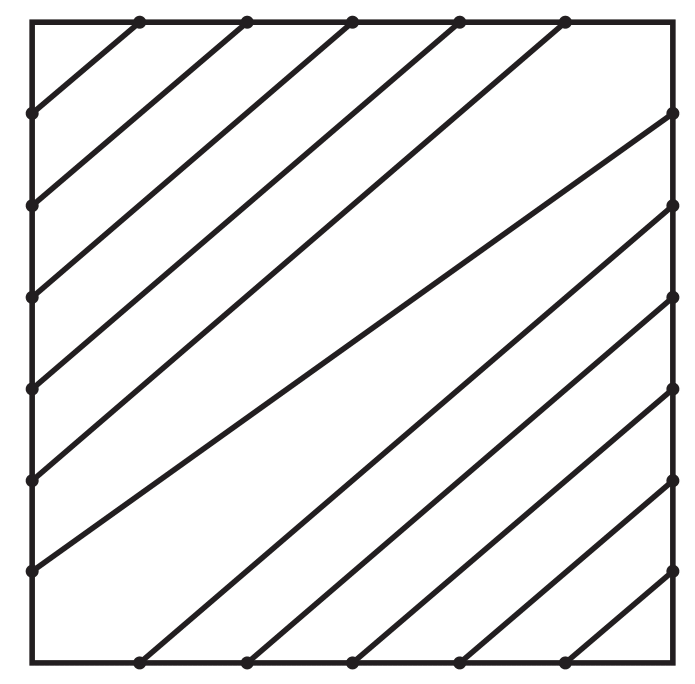}
    \caption{The standard embedding of the (5,6) torus link. }
    \label{fig:torus_link}
\end{figure}

While we provide some results for a more general family of torus graphs, targets of DNA nanostructure self-assembly are often surface meshes, and we will study in particular regular grids.

\begin{definition}\label{def:refinable}
Let $k,l \in \mathbb{N}$ with $l$ even and let $\Sigma$ be a surface. A \emph{$(k,l)$-regular $\Sigma$ grid} is an $l$-regular graph $G$ cellularly embedded on $\Sigma$ whose geometric dual is $k$-regular.
\end{definition}

We note that if there exists a $(k,l)$-regular grid on a torus, then it is either a $(4,4)$- or $(3,6)$-regular grid. Indeed, a straightforward application of the Euler characteristic of the torus shows that in this case $k$ and $l$ must lie on the hyperbola $2k = (k-2)l$. The only positive integer solutions of this equation for which $l$ is even are $(3,6)$ and $(4,4)$ (this can be checked by hand.) Additional information on $(k,l)$-regular grids on higher-genus surfaces can be found in \cite{GERMAN2002}. In the following, we refer to $(3,6)$-regular grids on the torus as \emph{triangular grids}, and $(4,4)$-regular grids on the torus as \emph{rectangular grids.}

\section{Unknotted A-trails and checkerboard-colorable embeddings}

Let $G$ be a graph embedded on the torus. Then $G$ is \emph{checkerboard-colorable} if the faces of $G$ can be properly $2$-colored. A \emph{checkerboard coloring} of $G$ is a particular proper $2$-coloring of the faces of $G$. Note that if $G$ is checkerboard-colorable, $G$ is necessarily Eulerian, but the converse does not hold (except on the plane). Throughout we will use $red$ and $blue$ to refer to the two colors used in a checkerboard coloring. Note that in a checkerboard-colored graph $G$, there are precisely two smooth transitions at a given vertex: the transition pairing those edges bounding red faces and the transition pairing those edges bounding blue faces. We call these transitions \emph{red smoothings} and \emph{blue smoothings} respectively.

In general, a checkerboard-colorable graph may not contain any A-trails (see Theorem \ref{thm:oddv} below for an infinite family of these.) Note that origami constructible embeddings are by definition smoothly Eulerian, but the converse may not hold. The first main result of this section will be the construction of a bijection between A-trails in a checkerboard-colored graph and certain pairs of quasitrees in associated embeddings. As a byproduct of this construction, we will see than any A-trail in a checkerboard-colored graph on the torus is contained in a disk. From this we obtain our second main result: any smoothly Eulerian checkerboard-colorable graph is origami constructible. Thus, the only obstruction to origami construction of a checkerboard-colorable graph on the torus is the existence of A-trails. We close by leveraging the bijection previously constructed to obtain a necessary condition on existence of such A-trails.

We begin by describing two embedded graphs associated to any checkerboard-colorable graph on the torus. Note that the red-graph and blue-graph defined below are, in the $4$-regular case, essentially equivalent to the red and white graphs of Andersen and Bouchet \cite{ANDERSEN1996232}.

\begin{definition} Let $G$ be a graph without loops embedded on the torus. Suppose $G$ is checkerboard-colorable and fix some such coloring in red and blue. The \emph{red-graph} of $G$, denoted $G_r$, is the thickened embedded graph obtained by placing a vertex in each red face of $G$, a vertex on each vertex of $G$, and connecting a vertex of $G_r$ in a face of $G$ to the vertices in the boundary-walk of this face (see Figure \ref{fig:red_graph}). We define the \emph{blue-graph} $G_b$ analogously. Let $x \in \{r,b\}$. A vertex of $G_x$ corresponding to a face of $G$ is a \emph{face-vertex}, a vertex of $G_x$ corresponding to a vertex of $G$ is a \emph{$G$-vertex}. We identify the $G$-vertices of $G_r$ and $G_b$ with the vertices of $G$ in the natural way. Let $v$ be a $G$-vertex of $G_x$. The \emph{$x$-neighborhood} of $v$, denoted $N_x(v)$, consists of $v$ together with its incident $x$-colored face-vertices and the edges between them (see Figure \ref{fig:red_nbhd}). Let $U \subseteq V(G)$. The subgraph of $G_x$ \emph{covered} by $U$ is $$G_x[U] = \bigcup_{v \in U} N_x(v)$$ where the union of two graphs is obtained by taking the union of their vertex sets and edge sets (see Figure \ref{fig:covering}). A \emph{covering subgraph} of $G_x$ is a subgraph $G_x[U]$ for some $U \subseteq V(G)$ that contains every face-vertex of $G_x$. A \emph{covering tree} is a covering subgraph of $G_x$ that is also a tree. A \emph{covering quasitree} is a covering subgraph of $G_x$ that is a quasitree. Note that every covering tree is necessarily a covering quasitree.
\end{definition}

\begin{figure}
    \centering
    \includegraphics[scale=6]{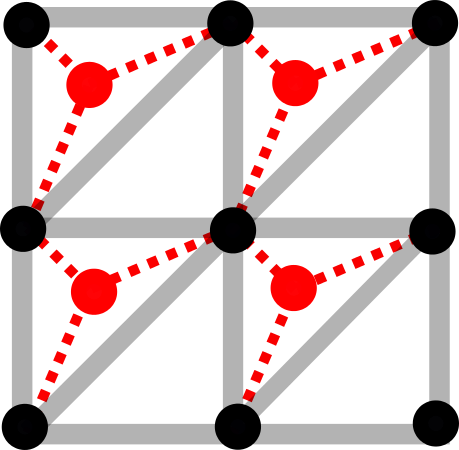}
    \caption{The red graph $G_r$ consists of the red face-vertices, the vertices of the triangular mesh, and the dashed edges.}
    \label{fig:red_graph}
\end{figure}

\begin{figure}
    \centering
    \includegraphics[scale=6]{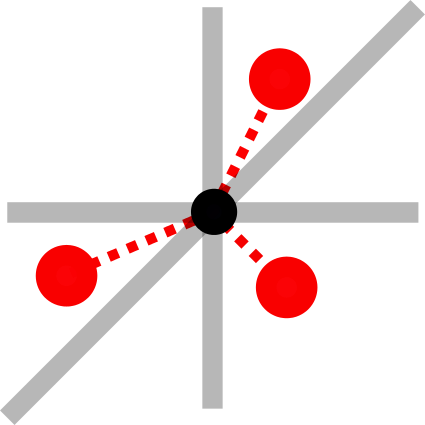}
    \caption{The red neighborhood $N_r(v)$ at the central vertex $v$ of Fig \ref{fig:red_graph} consists of $v$, the vertices in the red faces, and the edges between them and $v$.}
    \label{fig:red_nbhd}
\end{figure}

\begin{figure}
    \centering
    \includegraphics[scale=6]{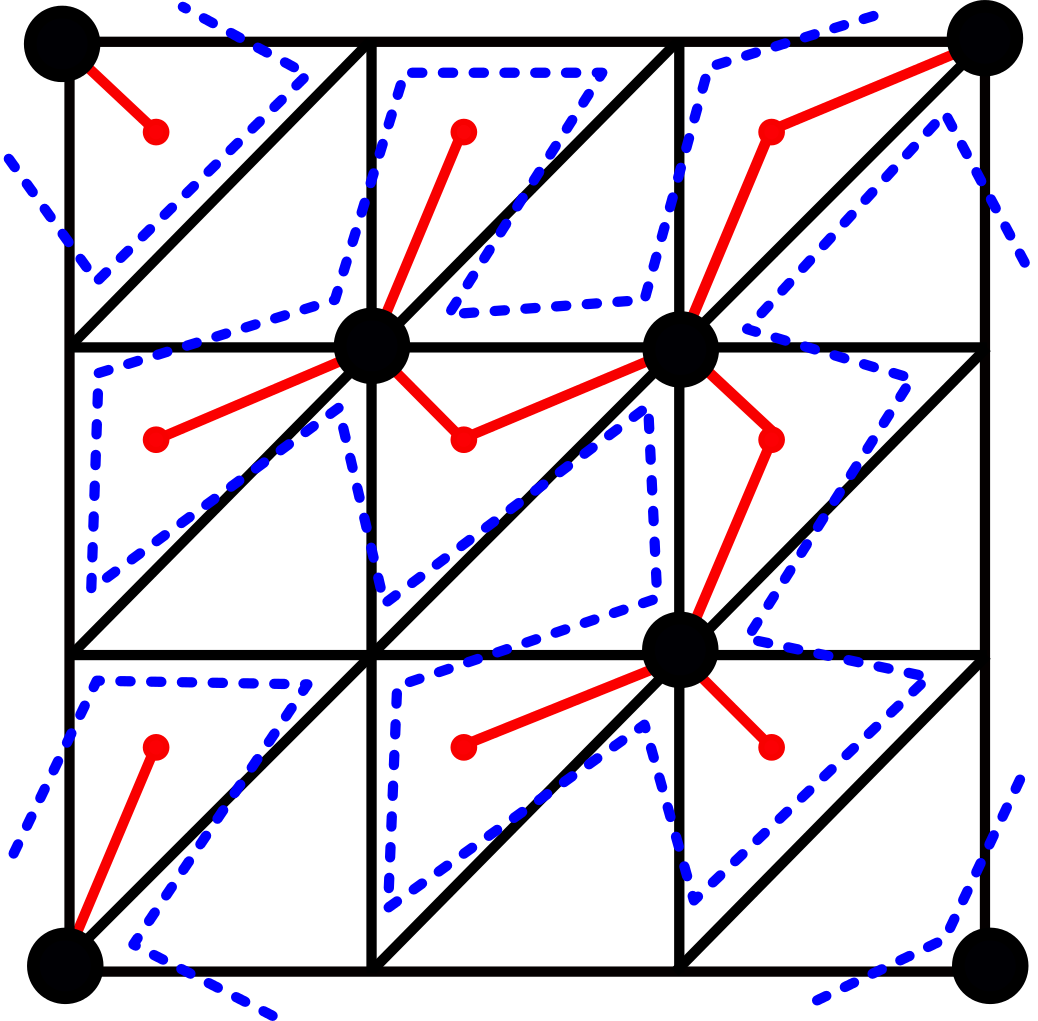}
    \caption{A covering tree $G_r[U]$ and the corresponding A-trail $\mathcal{S}_b[U]$. The vertex-set $U$ consists of the marked round vertices of the triangular grid. The A-trail $\mathcal{S}_b[U]$ is dashed.}
    \label{fig:covering}
\end{figure}

We can associate smooth transition systems of $G$ to subgraphs of $G$ and vice-versa.

\begin{definition} Let $G$ be a graph without loops embedded on the torus. Suppose $G$ is checkerboard-colorable and fix some such coloring in red and blue. Let $x \in \{r,b\}$. Let $x' \in \{r,b\} \setminus \{x\}$. Let $T$ be a smooth transition system of $G$. Denote by $T_r$ the vertices of $G$ at which $T$ is a red smoothing and denoted by $T_b$ the vertices of $G$ at which $T$ is a blue smoothing. Let $U \subseteq V(G)$. Denote by $\mathcal{S}_x[U]$ the smooth transition system that has $x$-smoothings at the vertices of $U$ and $x'$-smoothings at the vertices of $V(G) \setminus U$.
\end{definition}

Using this association, we have the following results.

\begin{lemma} \label{lem:retract} Let $G$ be a graph without loops embedded on the torus. Suppose $G$ is checkerboard-colorable and fix some such coloring in red and blue. Fix $x \in \{r,b\}$, $x' \in \{r,b\} \setminus \{x\}$, and let $U \subseteq V(G)$. Suppose $G_x[U]$ is a covering subgraph. Then $\mathcal{S}_{x'}[U]$ is homotopic to the boundary of $G_x[U]$, viewed as a thickened graph.
\end{lemma}

\begin{proof} We can see locally that the boundary of $G_r[U]$ is a deformation retract of $\mathcal{S}_b[U]$ (and the proof for blue graphs follows similarly). Consider a red face $f$. Let $v_f$ be the red vertex of $G_r$ in $f$. Let $w_1,\ldots,w_k$ be the $G$-vertices on the boundary of $f$, with $i_1< \ldots < i_j$ the indices corresponding to vertices in $U$. Each vertex $w_{i_m}$ is smoothed blue, and the paths $w_{i_m} w_{i_m+1} \cdots w_{i_{m+1}}$ are all smoothed red, and hence retract to the boundary of the neighborhood of $v_f$ as in Figure \ref{fig:retract}. Since $G$ is checkerboard-colored, no edge of $G$ lies on two distinct blue faces, hence we can piece together these local retractions to complete the proof.
\end{proof}

\begin{figure}
    \centering
    \includegraphics[scale=2]{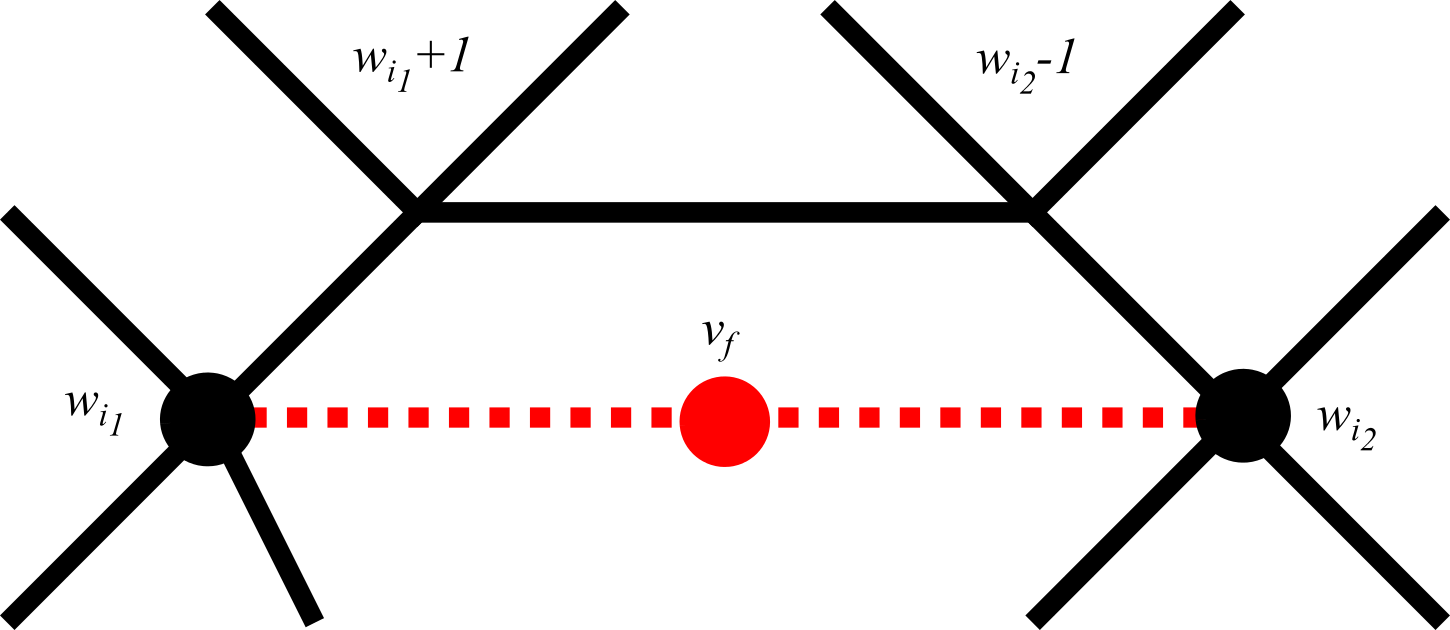} \hspace{2mm}
    \includegraphics[scale=2]{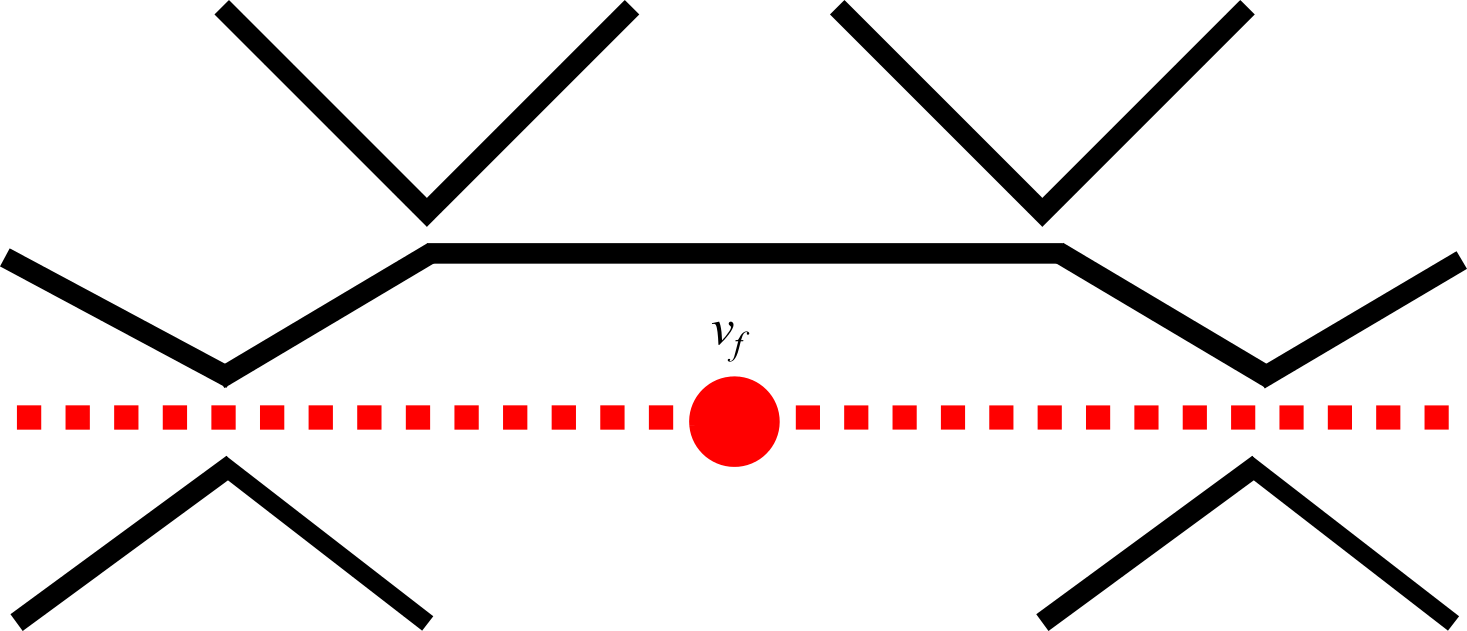}
    \caption{The retraction of $w_{i_m} w_{i_m+1} \cdots w_{i_{m+1}}$ from Lemma \ref{lem:retract}.}
    \label{fig:retract}
\end{figure}

The lemma above allows us to characterize A-trails in terms of pairs of red and blue covering quasitrees.

\begin{theorem} \label{thm:covtree} Let $G$ be a graph without loops embedded on the torus. Suppose $G$ is checkerboard-colorable and fix some such coloring in red and blue. Let $T$ be a smooth transition system of $G$. Then $T$ is an A-trail if and only if one of $G_r[T_b]$ or $G_b[T_r]$ is a covering tree.
\end{theorem}

\begin{proof} Without loss of generality, suppose $G_r[T_b]$ is a covering tree. Since $T$ is a smooth transition system by assumption, it suffices to show that $T$ consists of a single component. Since $G_r[T_b]$ is a covering subgraph, $T = \mathcal{S}_b[T_b]$ is homotopic to the boundary of $G_r[T_b]$. Since $G_r[T_b]$ is a tree, it has a single boundary component. Thus, $T$ consists of a single component.

Now suppose $T$ is an A-trail. Then both $G_r[T_b]$ and $G_b[T_r]$ are covering subgraphs. Since $T = \mathcal{S}_r[T_r] = \mathcal{S}_b[T_b]$, it follows that $T$ is homotopic to both the boundary of $G_b[T_r]$ and the boundary of $G_r[T_b]$. Suppose $G_r[T_b]$ is not a tree, i.e. contains some cycle $C$. If $C$ is null-homotopic, then the boundary of $G_r[T_b]$ has more than one component, contradicting the assumption that $T$ is an A-trail. Thus, $C$ must complete some meridional and/or longitudinal rotations. We claim that $G_r[T_b]$ must contain some path $P \not \subseteq C$ whose endpoints lie on $C$ and, when the torus is cut along $C$ into an annulus $P$ connects the two boundary components of the annulus. Indeed, otherwise the boundary component of $G_r[T_b]$ containing one side of an edge $e \in C$ cannot contain the other side, again contradicting the assumption that $T$ is an A-trail. Since the vertex-sets of $G_r[T_b]$ and $G_b[T_r]$ are disjoint, as an embedded graph $G_b[T_r]$ is contained in the complement of $C \cup P$ on the torus, which is a disk. Thus, since $G_b[T_r]$ is embedded in a disc and has a single boundary component, $G_b[T_r]$ must be a tree.
\end{proof}

The proof of the above theorem yields the following corollary.

\begin{corollary} A smoothly Eulerian graph $G$ on the torus is checkerboard-colorable if and only if any A-trail in $G$ is contained in some open disk (which may depend on the A-trail).
\end{corollary}

\begin{proof} Suppose $G$ is checkerboard-colorable. Let $T$ be an A-trail of $G$. From the proof of Theorem \ref{thm:covtree} we have that $T$ is contained in the complement of $C \cup P$ on the torus, which is an open disk.

Now suppose any A-trail in $G$ is contained in some open disk. Let $T$ be one such A-trail, contained in the open disk $D$. By the Jordan-Sch\"{o}nflies theorem, $T$ separates $D$ into an interior $I$ and exterior $E$ for which it is the common boundary. By the definition of smooth transitions, each face of $G$ is either contained entirely in $E$ or entirely in $I$. Color those faces in $I$ red and color the faces in $E$ blue. We claim that this properly 2-colors the faces of $G$. Indeed, give $T$ an orientation. Since $T$ is the common boundary of both $I$ and $E$, $I$ is either consistently to the left of $T$ or consistently to the right of $T$ under that orientation. Without loss of generality we may assume the former. Then each edge of $G$ has a red face to its left, and a blue face to its right, and therefore the faces of $G$ are properly 2-colored.
\end{proof}

It therefore follows that, for checkerboard-colorable graphs, being smoothly Eulerian and being origami constructible are equivalent properties.

\begin{theorem} Let $G$ be a checkerboard-colorable graph on the torus. Then $G$ is origami constructible (regardless of the embedding of the torus in space) if and only if $G$ is smoothly Eulerian.
\end{theorem}

\begin{proof} Origami constructible implies smoothly Eulerian for any embedded graph by definition.  Now suppose $G$ is smoothly Eulerian, and let $T$ be an A-trail of $G$. since $G$ is checkerboard-colorable, $T$ is contained in a disk, and so by Jordan-Sch\"{o}nflies bounds a disk on the torus, which remains a disk when the torus is embedded in space. Therefore, $T$ is unknotted.
\end{proof}

Thus, to demonstrate that a checkerboard-colorable torus graph is origami constructible we need only demonstrate that it contains an A-trail. We close this section by recording the following necessary condition for existence of such A-trails, to be used in the following section in the construction of an infinite family of origami-constructible triangulations of the torus.

\begin{corollary} \label{cor:nec} Let $G$ be a checkerboard-colored graph on the torus. Let $k_r$ and $k_b$ be the total number of red faces and total number of blue faces respectively. Suppose each vertex of $G$ is incident to an odd number of red faces and an odd number of blue faces. If $G$ contains an A-trail, then one of $k_r$ or $k_b$ must be odd.
\end{corollary}

\begin{proof} We begin by showing that trees $G_x[A]$ for $A \subseteq V(G)$ have an odd number of face vertices. Let $A$ be a set of $G$-vertices. If $|A| = 1$, then the number of face-vertices of $G_r[A]$ is $3$ and hence odd. Now suppose $|A| > 1$. Let $P$ be a maximal path in $G_r[A]$. Let $v$ be the last $G$-vertex of $P$. By induction, the number of face-vertices of $G_r[A \setminus \{v\}]$ is odd. But since $G_r[A]$ contains no cycles, $N_r(v)$ contains two face-vertices not in $G_r[A \setminus \{v\}]$, and hence $G_r[A]$ has an odd number of face-vertices.

To complete the argument, we note that if $G$ contains an A-trail, the previous theorem implies that either the red graph or the blue graph contains a covering tree, which by the previous argument has an odd number of face-vertices.
\end{proof}

\section{Origami construction of triangular torus grids} 

The main result of this section is a characterization of existence of A-trails in a certain infinite family of triangular torus grids. As a result, we will have constructed an infinite family of origami constructible triangular grids as well as an infinite family of non-origami constructible (indeed, non-smoothly Eulerian) triangular grids. We give an explicit construction of unknotted A-trails in the origami constructible family.

Representations of triangular grids in the torus have all been constructed by Altshuler \cite{ALTSHULER1973201} as follows. Let $G$ be a triangular grid. A \emph{straight-ahead} walk in $G$ is a walk $C$ so that the edges paired at each $v$ in $C$ have two other edges between them in the cyclic order at $v$. Altshuler has shown that beginning at any vertex and walking straight-ahead in any direction will trace out a straight-ahead cycle ending at the first vertex. Let $a_1^1$ be a vertex of $G$. Let $C= a_1^1, a_2^1, \ldots, a_n^1$ be one of the three straight-ahead cycles through $a_1^1$. Let $a_1^2$ be the vertex two steps counterclockwise from $a_2^1$ in the cyclic order at $a_1^1$, and let $a_1^1 a_1^2 \cdots a_r^1 a_m^1$ be the straight ahead path starting at $a_1^1$, going through $a_1^2$ which (necessarily) ends at a vertex $a_m^1$ on $C$. Then $r = v/n$ is an integer and $G$ can be represented as in Figure \ref{fig:Alt_gen} \cite{ALTSHULER1973201}. Note that the identification of the vertical sides is the same as usual, but the top and bottom are identified as labeled by the vertices $a_i^j$, i.e. after gluing the vertical sides together, one performs a twist before gluing the top to the bottom. Following Altshuler's notation, we denote this representation of $G$ by $T_m^{v,r}$. We call all such representations \emph{Altshuler representations.}

\begin{figure}[ht!]
    \centering
    \includegraphics[scale=0.15]{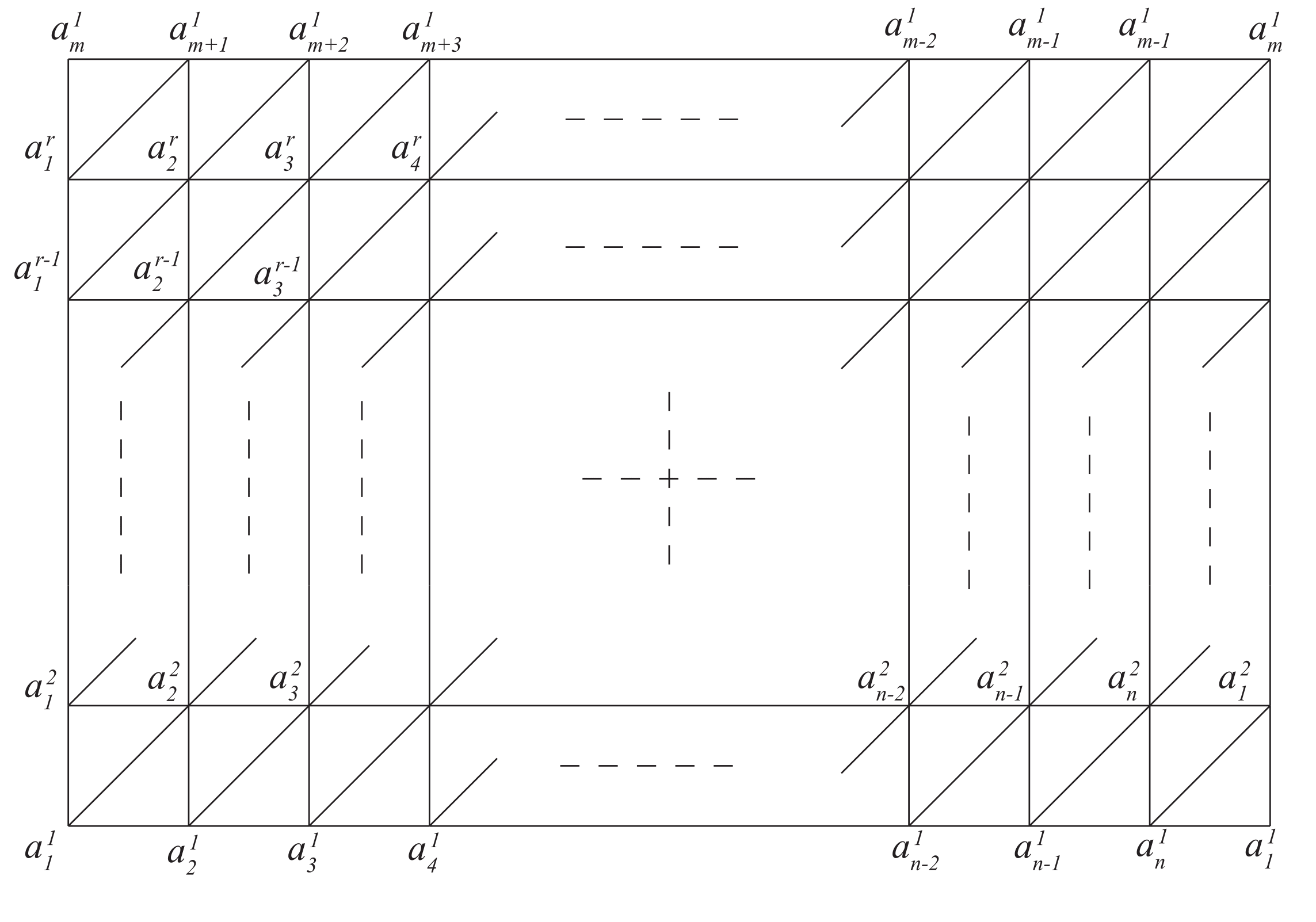}
    \caption{Altshuler's general representation of a $(3,6)$-regular triangulation of the torus.}
    \label{fig:Alt_gen}
\end{figure}

Directly from Altshuler's representation we have that any triangular grid is checkerboard-colorable. Hence, a triangular grid is origami constructible if and only if it is smoothly Eulerian. In the remainder of this section we will character smoothly Eulerian grids for a large class of triangular grids. Let $G$ be a triangular grid with $v$ vertices. We say $G$ is \emph{straight-ahead Hamiltonian (SAH)} if every non-loop straight ahead cycle is Hamiltonian. Note that this is equivalent to the statement that every Altshuler representation of $G$ has $r \in \{1,v\}$.

We can now characterize the smoothly Eulerian triangular grids that are not SAH.

\begin{theorem} \label{thm:oddv} Let $G$ be a triangular grid with $v$ vertices that is not SAH. Then $G$ is smoothly Eulerian if and only if $v$ is odd.
\end{theorem}

\begin{proof} Suppose $G$ is a triangular grid with $v$ vertices containing an A-trail $T$. Fix a checkerboard coloring and Altshuler representation $T_m^{v,r}$ of $G$ with $1<r<v$ (which is possible since $G$ is not SAH). Note that each vertex of $G$ is incident to three red faces and three blue faces. By Corollary \ref{cor:nec}, either $k_r$ (the total number of red faces) or $k_b$ (the total number of blue faces) is odd. If $f$ is the number of faces of $G$, from Altshuler's representation we have that $k_r = k_b = f/2 = v$. Thus, $v$ is odd.

Now suppose $v$ is odd. We have $r \geq 3$. Suppose $r=3$. An A-trail in this case is given by the transitions in the top row and bottom two rows of Figure \ref{fig:trichar}, with the break in the pattern in the bottom row occurring at the vertex $a_m^1$. It is straightforward to check that no matter the value of $m$, this yields an A-trail. Now suppose $r > 3$. Since $r$ is odd, we can fill the rows between the first row and the bottom two rows of Figure with the pattern demonstrated in Figure \ref{fig:trichar}.
\end{proof}

\begin{figure}[h]
    \centering
    \includegraphics[scale=.1]{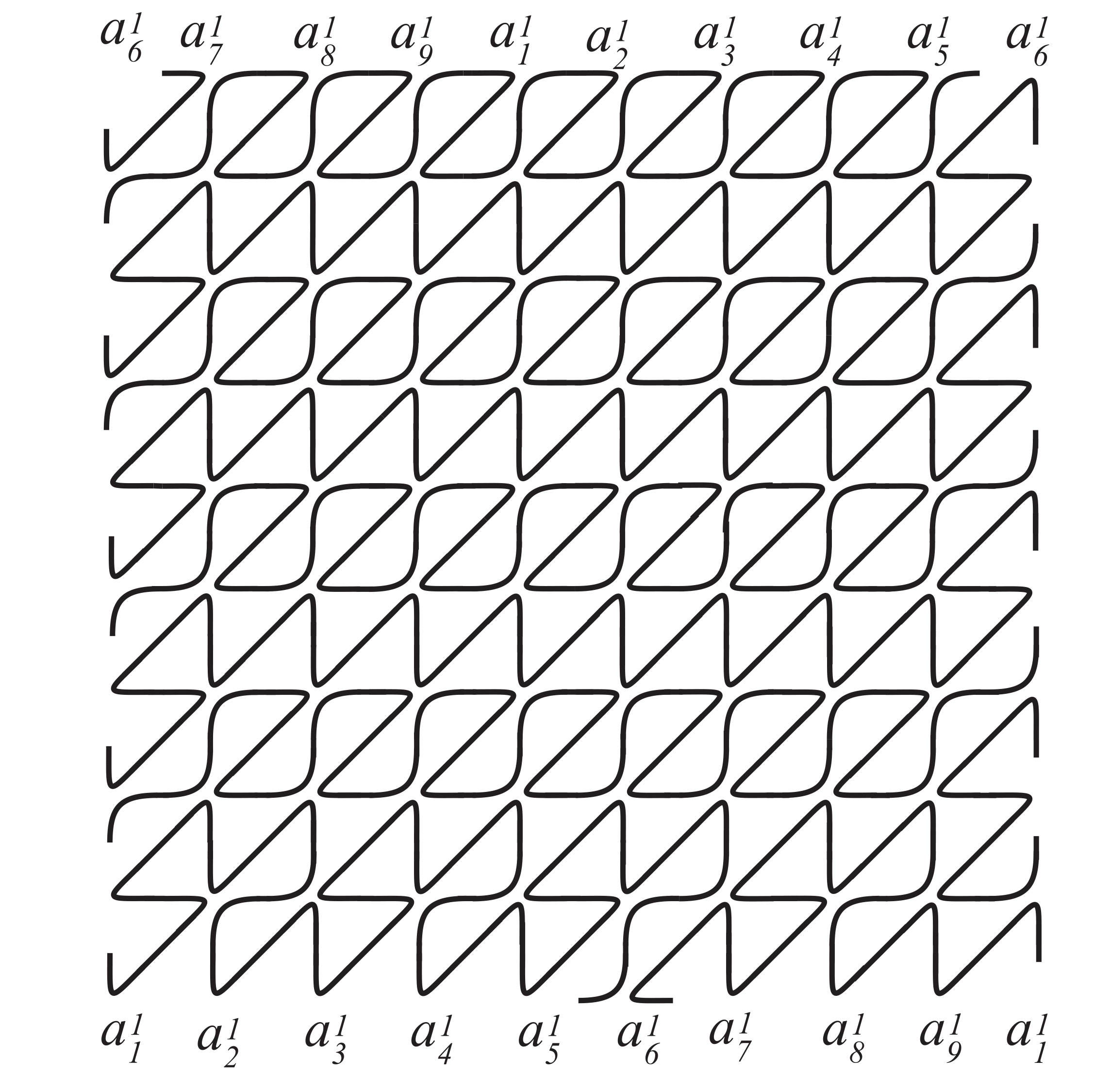}
    \caption{The A-trail constructed in Theorem \ref{thm:oddv} for $T_6^{81,9}$.}
    \label{fig:trichar}
\end{figure}

In light of the application, we record the following corollary.

\begin{corollary} A non-SAH triangular grid with $v$ vertices is origami constructible if and only if $v$ is odd.
\end{corollary}

The preceding theorem and corollary yield both an infinite family of origami constructible torus triangulations (non-SAH grids with an odd number of vertices) as well as an infinite family of non-origami constructible, non-smoothly Eulerian torus triangulations. Lastly, we note that existence of A-trails in SAH grids can depend on $m$: it is straightforward to check that $T_2^{5,1}$ is smoothly Eulerian while $T_1^{5,1}$ is not.

\section{Origami construction of rectangular torus grids.} 

The rectangular torus grids have similar representations to the Altshuler representation of triangular grids \cite{GERMAN2002}. However, since not all rectangular grids are checkerboard-colorable, we will need to restrict our attention to certain spatial embeddings. As an example of what can happen if we do not, we provide in Figure \ref{fig:1vert} two embeddings of a one-vertex graph with two loops on a torus. On the level of cellularly embedded graphs, these embeddings are equivalent and we can move from one embedding to the other by Dehn twists, i.e. cutting the torus open along some simple closed curve with tubular neighborhood, twisting, and gluing the torus back together (see e.g. \cite{HUMPHRIES} for details). With the torus embedded in $\mathbb{R}^3$, however, these twists affect the isotopy classes of curves on the torus. Indeed, in the first embedding all A-trails are knotted, while in the second all A-trails are unknotted.

\begin{figure}[ht!]
    \centering
   
    \begin{subfigure}[b]{0.5\columnwidth}
        \centering
        \includegraphics[scale=0.1]{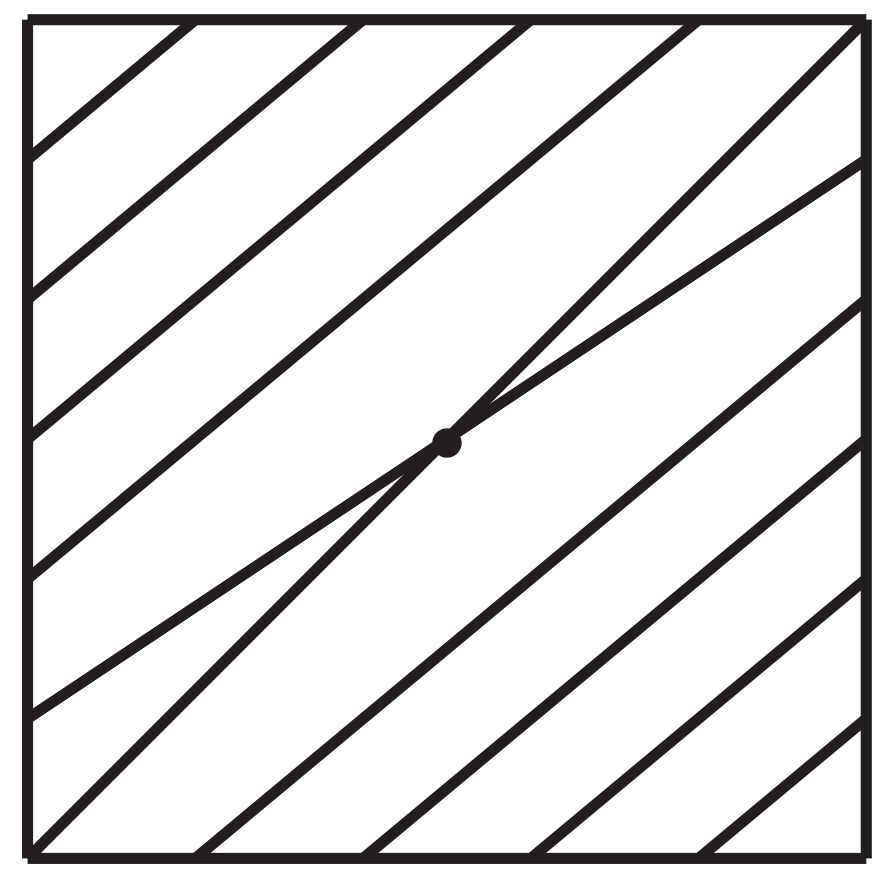}
        \caption{A $(4,4)$-regular grid with $1$ vertex.}
        \label{subfig:1vert_1}
    \end{subfigure}%
    \begin{subfigure}[b]{0.5\columnwidth}
        \centering
        \includegraphics[scale=0.1]{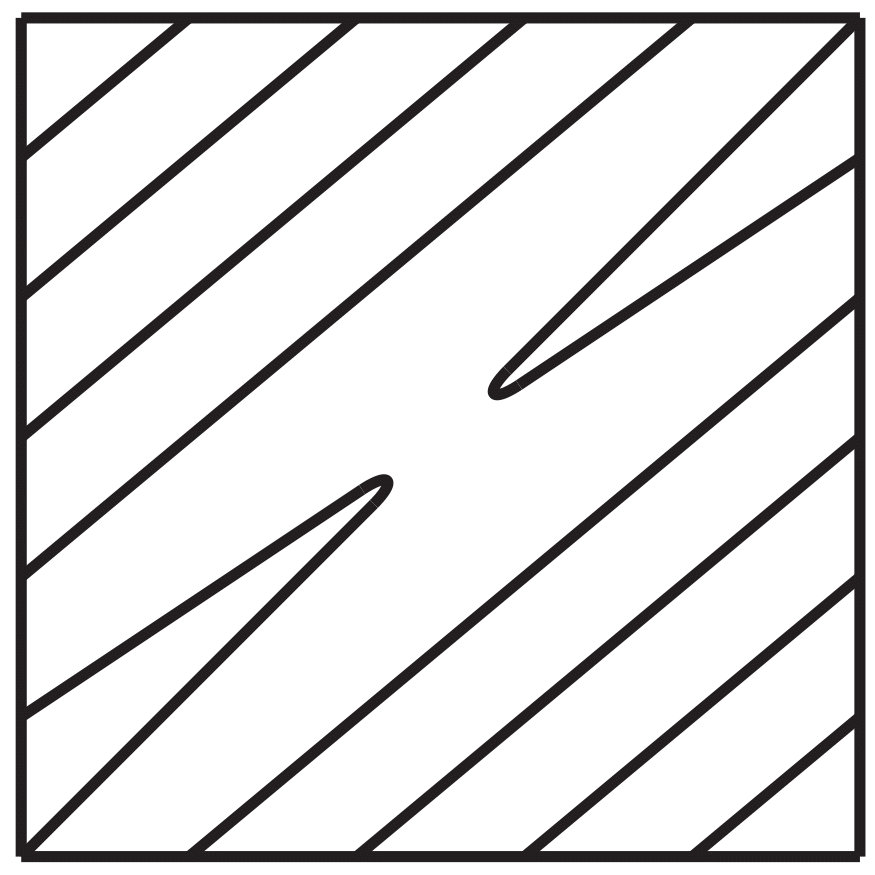}
        \caption{The first A-trail, forming a $(4,3)$ knot.}
        \label{subfig:1vert_2}
    \end{subfigure}
    
    \begin{subfigure}[b]{0.5\columnwidth}
        \centering
        \includegraphics[scale=0.1]{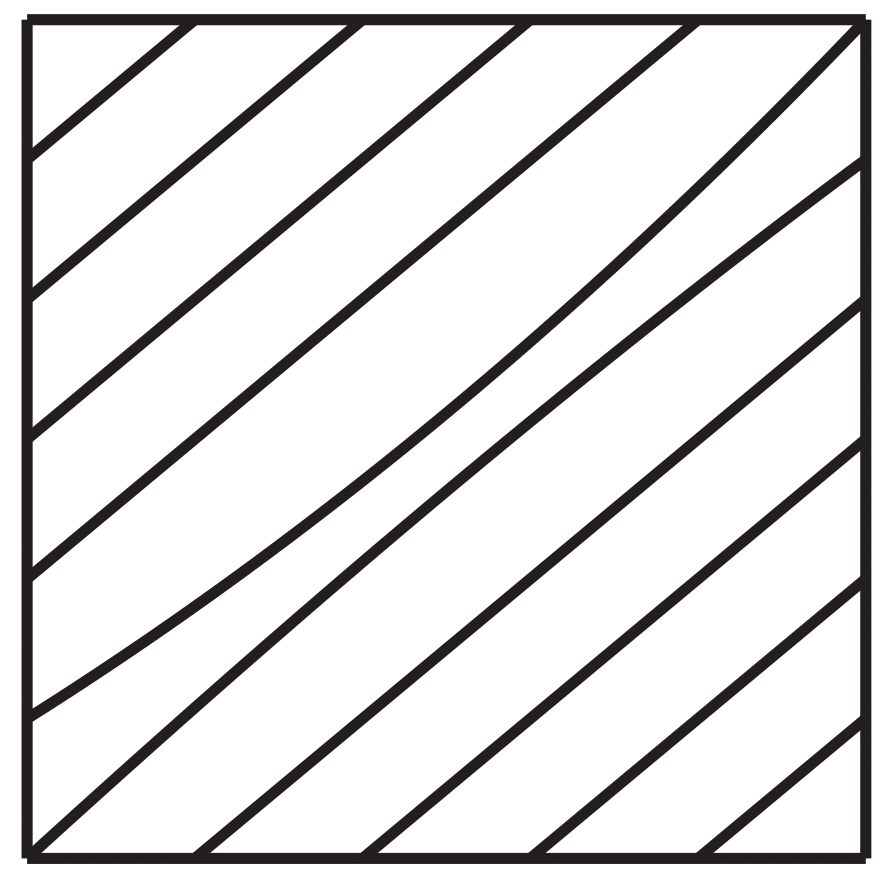}
        \caption{The second A-trail, forming a $(6,5)$ knot.}
        \label{subfig:1vert_3}
    \end{subfigure}%
    \begin{subfigure}[b]{0.5\columnwidth}
        \centering
        \includegraphics[scale=0.1]{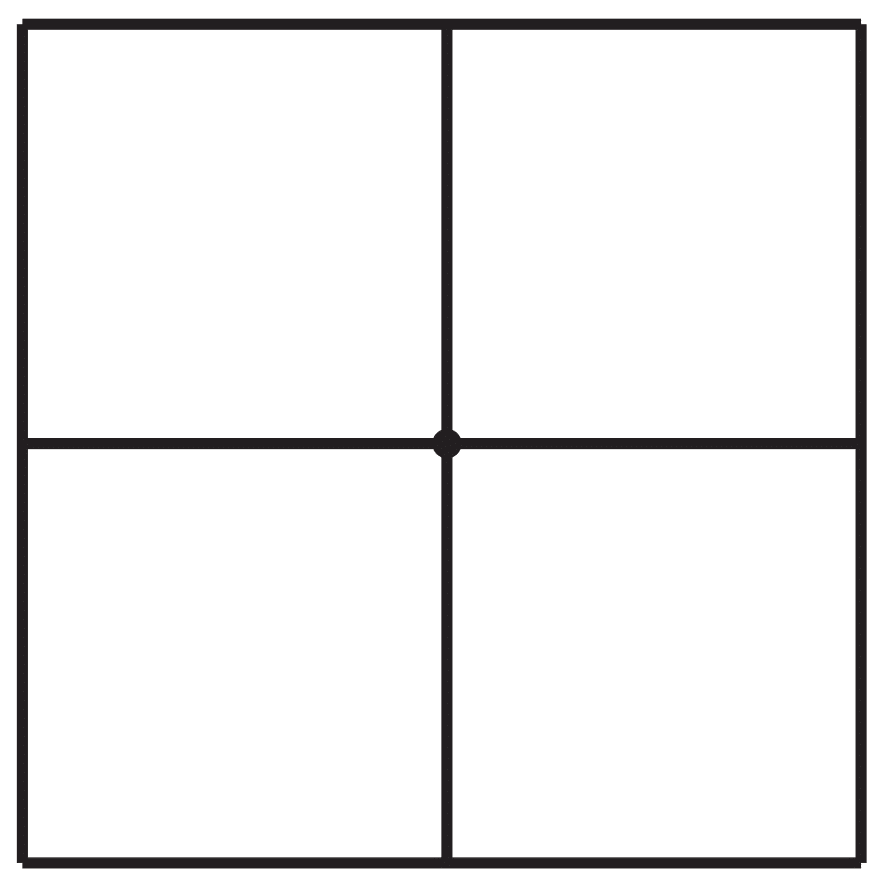}
        \caption{A different embedding of the same grid, \phantom{TTT} here both A-trails are unknotted.}
        \label{subfig:1vert_4}
    \end{subfigure}
    \caption{Two embeddings of a $(4,4)$-regular grid with one vertex. In the first embedding, every A-trail forms a knot; in the second, every A-trail is unknotted.}
    \label{fig:1vert}
\end{figure}

\newpage

\begin{definition} For $i,j>1$, the \emph{rectangular torus grid of width $i$ and height $j$} is the embedding $R_{i,j}$ of $C_i \times C_j$ on the torus as in Figure \ref{fig:c7xc4}, with the torus standardly embedded in $\mathbb{R}^3$.
\end{definition}

\begin{figure}[ht!]
\centering
\includegraphics[scale=.15]{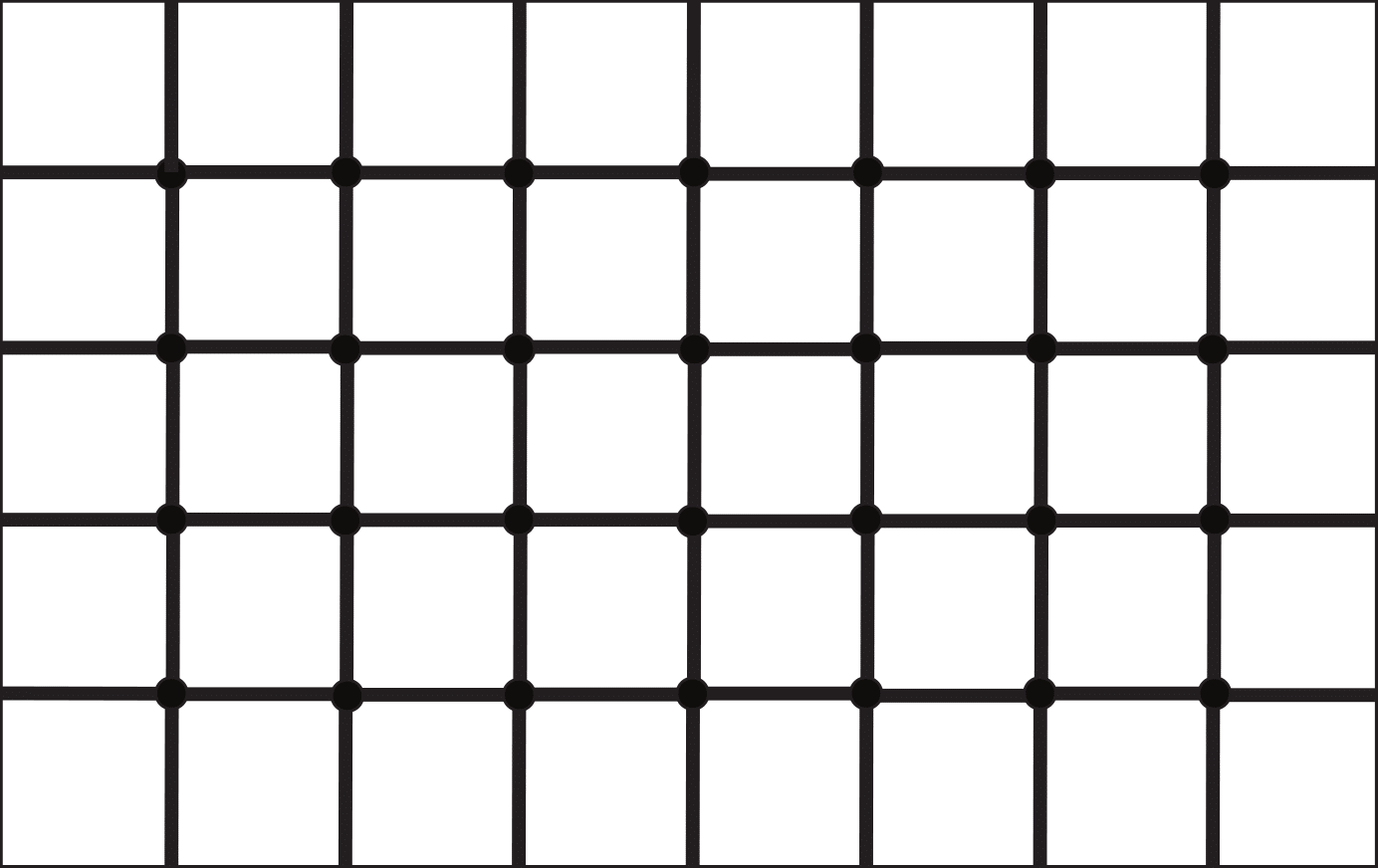}
\caption{The rectangular torus grid $R_{7,4}$.}
\label{fig:c7xc4}
\end{figure}

While not every rectangular grid $R_{i,j}$ is checkerboard-colorable, every one contains an unknotted A-trail.

\begin{theorem}\label{thm:unknot_rect}
Every grid $R_{i,j}$ contains an unknotted A-trail.
\end{theorem}

\begin{proof} We give examples of three cases in Figure \ref{fig:unknot}, and leave generalization of these to the reader. It is most convenient here to place the vertices of $R_{i,j}$ on the boundaries of the fundamental square representing the torus (so, in particular, all four corners are identified to a vertex.)
\end{proof}

\begin{figure}[ht]
\centering
    \begin{subfigure}[b]{0.3\columnwidth}
        \centering
        \includegraphics[scale=.1]{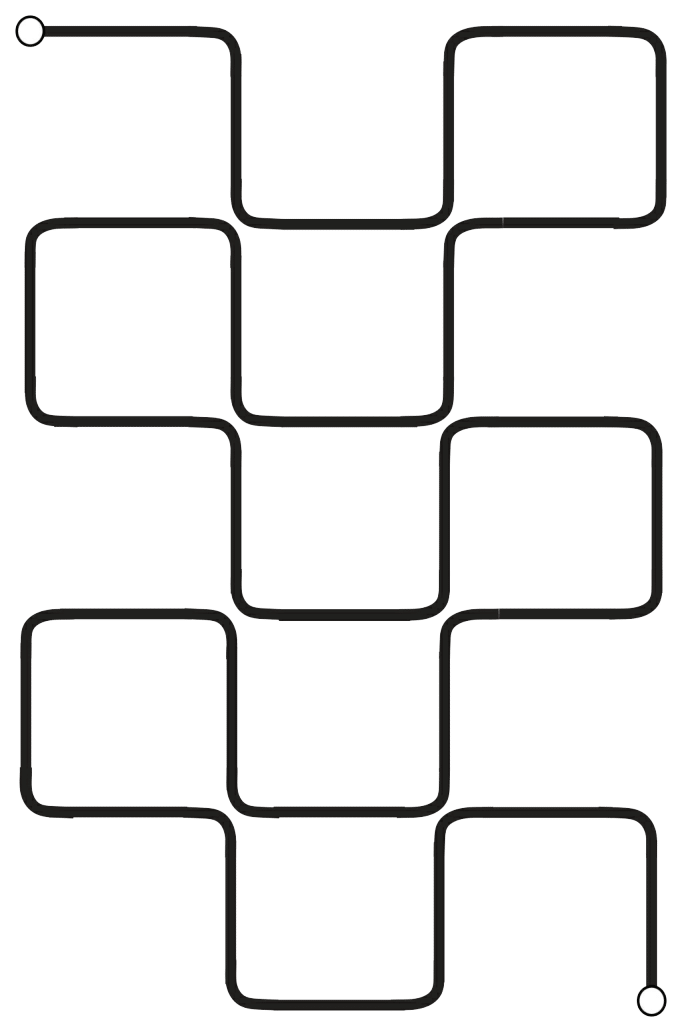}
        \caption{$3\times5$}
        \label{subfig:OddParity}
    \end{subfigure}%
    \begin{subfigure}[b]{0.3\columnwidth}
        \centering
        \includegraphics[scale=.1]{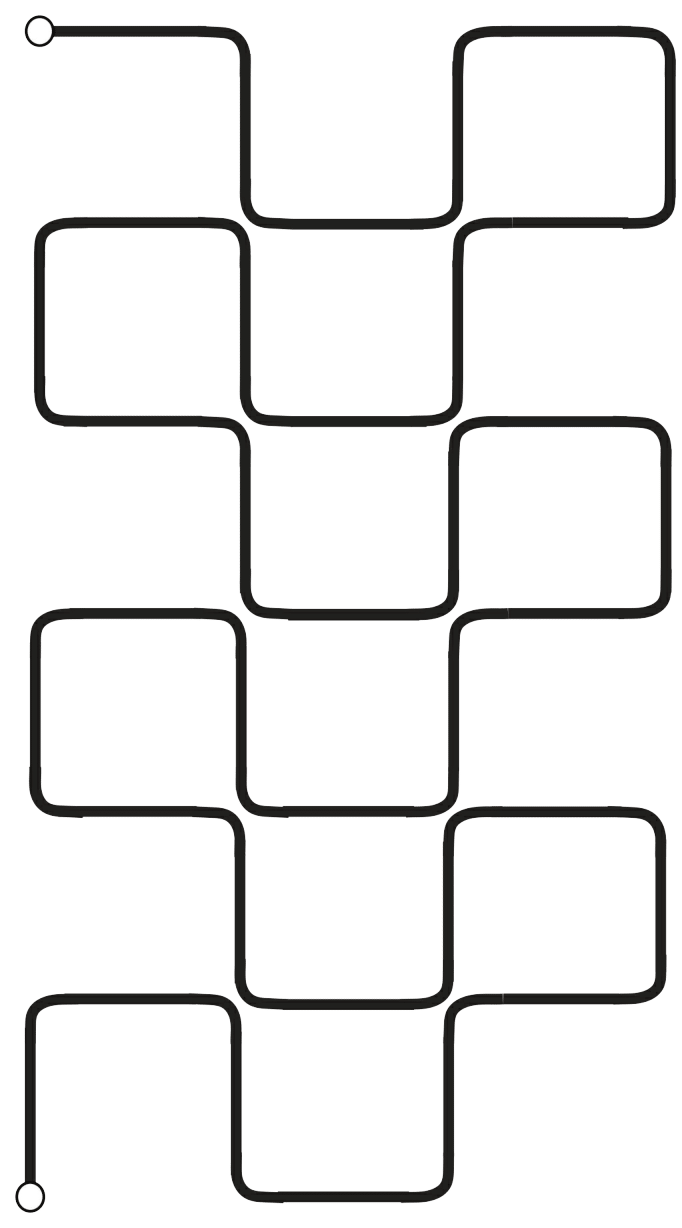}
        \caption{$3\times6$}
        \label{subfig:OddEven}
    \end{subfigure}%
    \begin{subfigure}[b]{0.3\columnwidth}
        \centering
        \includegraphics[scale=.1]{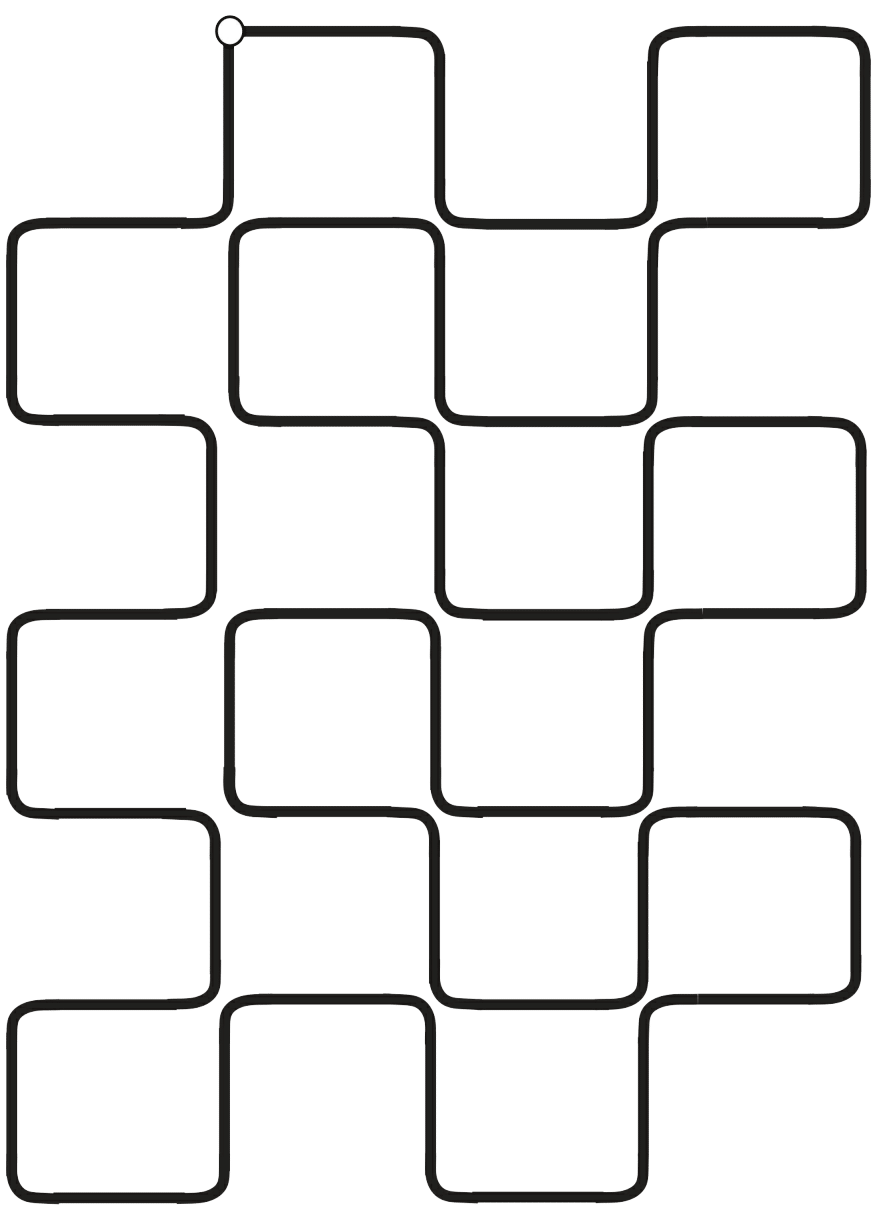}
        \caption{$4\times6$}
        \label{subfig:EvenParity}
    \end{subfigure} 
\caption{Examples of Unknotted A-trails in grids of each parity.}
\label{fig:unknot}
\end{figure}

Note that if $i$ or $j$ is odd, $R_{i,j}$ is not checkerboard-colorable, and hence any A-trails must complete a longitudinal or meridional rotation. The unknotted A-trails presented above for these cases rely, therefore, on being embedded on an unknotted torus. If $i$ and $j$ are both even, $R_{i,j}$ is checkerboard colorable, and the assumption that the torus be unknotted is not necessary.

Since many rectangular grids are not checkerboard-colorable, they may contain A-trails forming nontrivial torus knots as well as circuit decompositions forming nontrivial torus links. A natural question to ask is whether every torus link can be constructed from a smooth circuit decomposition of some $R_{i,j}$.

\begin{theorem} \label{thm:alltoruslinks} Let $L$ be an oriented link embedded on an (unknotted) torus $T$. There exists a rectangular torus grid $R_{i,j}$ having a smooth transition system $T$ isotopic to $L$.
\end{theorem}

\begin{proof} Suppose $L$ is the $(p,q)$ torus link. If $p,q > 0$, then the smooth circuit decomposition of $R_{p,q}$ given in Figure \ref{fig:relprime} (oriented left-to-right, bottom-to-top) is isotopic  to $L$. Appropriately rotating this construction (and reversing orientation if necessary) handles the case where $p$ or $q$ is negative.
\end{proof}

\begin{figure}[!ht]
    \centering
    
    \begin{subfigure}[b]{0.5\columnwidth}
    \centering
    \includegraphics[scale=.1]{relprime0-1.png}
    \caption{$R_{7,4}.$}
    \label{subfig:relprime0}
    \end{subfigure}%
    \begin{subfigure}[b]{0.5\columnwidth}
    \centering
    \includegraphics[scale=.1]{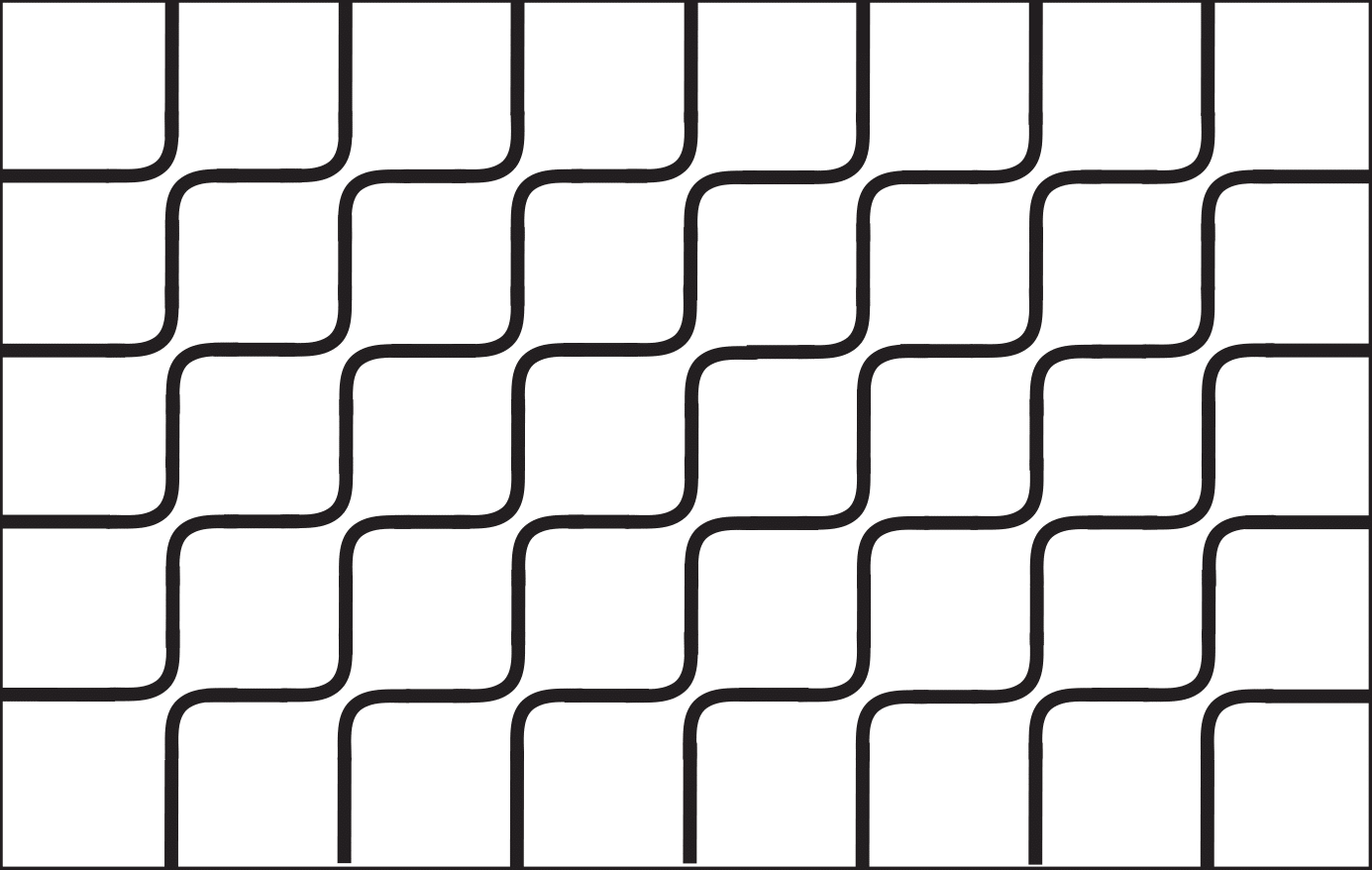}
    \caption{The $(7,4)$ torus link.}
    \label{subfig:relprime1}
    \end{subfigure}
\caption{Finding a $(7,4)$ torus knot on $R_{7,4}$.}
\label{fig:relprime}
\end{figure}

\begin{figure}[h]
    \centering
        \includegraphics[scale=.1]{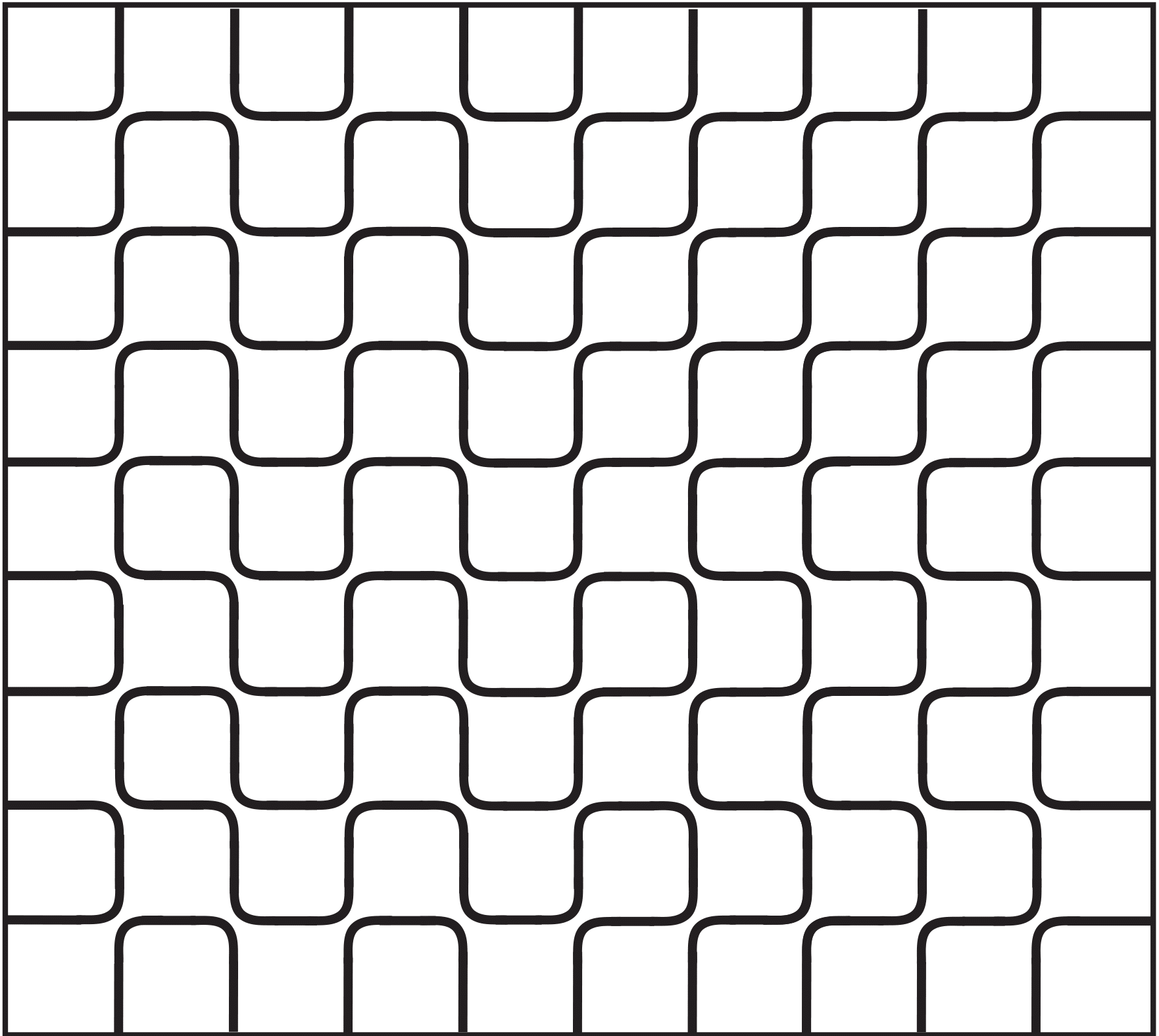}
        \caption{A circuit decomposition of $R_{9,8}$ forming a $(5,4)$ torus knot.}
    \label{fig:circ_decomp}
\end{figure}

We note that each grid $R_{i,j}$ can contain many links other than the $(i,j)$ torus link. For example, we give in Figure \ref{fig:circ_decomp} a circuit decomposition of $R_{9,8}$ forming a $(5,4)$ torus knot. Note that this construction can be straightforwardly modified to produce other torus links by adjusting the number of ``loops'' along the sides (each loop reduces the number of times the resulting link completes meridional or longitudinal rotations by $2$.) Essentially one is shifting the standard representation of of the $(m,n)$-torus link for $m<i$, $n<j$ (with $i-m$ and $j-n$ even) to the $m \times n$ upper-right corner of $R_{i,j}$.

\section{Origami construction of higher-genus grids} While the results above are specific to the torus, using a gluing operation on grids and transition systems we can construct families of grids on higher-genus surfaces containing unknotted A-trails. While we will focus on grids constructed from triangular and rectangular torus grids, we present the basic constructions in more generality, as they can be used for analyzing other target structures. For convenience, we will say a face of an embedded graph is \emph{cyclic} if its boundary walk is a cycle.

\begin{definition} \label{def:composite} Let $\{G_i \}_{i=1}^n$ be a collection of torus grids $G_i$. A \emph{gluing set} for $\{G_i\}$ is a collection $\{(f_{i,1}, f_{i,2}, \pi_i)\}_{i=1}^{n-1}$ such that $f_{i,1}$ is a cyclic face of $G_i$, $f_{i,2}$ is a cyclic face of $G_{i+1}$, $f_{i,2} \cap f_{i+1,1} = \emptyset$, and $\pi_i: \partial f_{i,1} \to \partial f_{i,2}$ is an isomorphism of cycle graphs. Remove the interior of each face $f_{i,j}$ and glue $\partial f_{i,1}$ to $\partial f_{i,2}$ according to the isomorphism $\pi_i$. The resulting embedded graph $G$ is a $\{G_i\}$-composite grid of genus $n$. We also say simply that $G$ is a \emph{$\{G_i\}$-composite by $\{ (f_{i,1}, f_{i,2}, \pi_i) \}$}. The individual grids $G_i$ are called \emph{component grids.}
\end{definition}

Composite grids of a fixed collection $\{G_i\}$ are not unique; distinct grids can be formed by choosing different faces to identify. We also note that knottedness of A-trails on composite grids may depend on the embedding of the resulting genus $n$ surface $\Sigma$ in $\mathbb{R}^3$. We will therefore assume throughout that $G$ embedded on $\Sigma$ is embedded in a standard way, as in Figure \ref{fig:standard-torus}.

\begin{figure}
    \centering
    \includegraphics[scale=5]{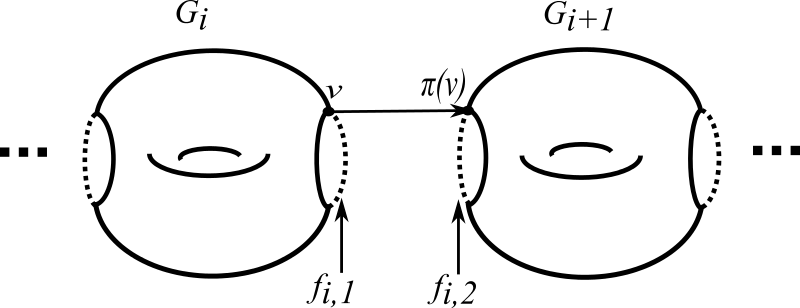}
    \caption{Standard embedding of a composite grid.}
    \label{fig:standard-torus}
\end{figure}

We record some basic properties of composite grids below.

\begin{lemma} Let $G$ be a $\{G_i\}_{i=1}^n$-composite by $\{ (f_{i,1}, f_{i,2}, \pi_i) \}$.
\begin{enumerate}
    \item $G$ is cellularly embedded on the $n$-torus.
    \item If $G_i$ is a triangular (resp. rectangular) grid for each $i$, then $G$ is a triangular (resp. rectangular) grid.
    \item If $G_i$ is Eulerian for each $i$,  then $G$ is Eulerian.
\end{enumerate}
\end{lemma}

%%%%%%%%%%%%%%%%%%% CONNECTED SUM OF TRANSITION SYSTEMS %%%%%%%%%%%%%%%%%%%%%

To analyze A-trails in composite grids, we will require the following construction of composite A-trails.

\begin{definition}
For an embedded graph $G$ and vertex $v$ lying on a face $f$, we say a transition at $v$ is an \emph{$f$-transition} if it pairs the half-edges of $f$ incident to $v$. Let $G_1$ and $G_2$ be torus graphs and let $T_i$ be a smooth transition system on $G_i$. We say that $T_1$ and $T_2$ are \emph{compatible} if there exist cyclic faces $f_1$ of $G_1$ and $f_2$ of $G_2$ and a isomorphism $\pi: \partial f_1 \to \partial f_2$ such that $T_1$ is an $f_1$-transition at $v \in \partial f_1$ if and only if $T_2$ is not an $f_2$ transition at $\pi(v)$. In this case the triple $(f_1,f_2,\pi)$ \emph{exhibits compatibility} of $T_1$ and $T_2$. \end{definition}

Given compatible smooth transition systems, by ``forgetting'' when one transition system pairs half-edges along the face being glued we can construct a smooth transition system in the resulting composite. %(see Figure \ref{fig:comp_smooths}).

\begin{definition} Under the notation of the previous definition, and assuming $(f_1,f_2,\pi)$ exhibits compatibility of $T_1$ and $T_2$, let $G$ be the composite of $G_1$ and $G_2$ formed by gluing $f_1$ and $f_2$ according to $\pi$. Let $C$ be the cycle resulting from gluing $f_1$ and $f_2$. Let $v \in C$ with $v_1$ and $v_2$ its preimages in $f_1$ and $f_2$ respectively. Let $t_1$ be the $T_1$ transition at $v_1$ and let $t_2$ be the $T_2$ transition at $v_2$. The \emph{connected sum} $t_1 \# t_2$ of $t_1$ and $t_2$ is the transition at $v$ in $G$ constructed as follows (see Figure \ref{fig:comp_smooths}). Any pair of half edges in $t_1$ or $t_2$ not containing a half-edge in $f_1$ or $f_2$ respectively remains a pair in $t_1 \# t_2$ at $v$. By compatibility, precisely one of the transitions $t_i$ pairs the half-edges $h$ and $h'$ of $f_i$ with half-edges $s$ and $s'$ not on $f_i$. In $t_1 \# t_2$, we pair $s$ and $s'$ with the half-edges of $C$ whose preimages in $f_i$ are $h$ and $h'$. By construction, $t_1 \# t_2$ is a smooth transition at $v$.
\end{definition}

\begin{figure}[ht!]
    \centering
    \begin{subfigure}[b]{0.5\columnwidth}
       \centering
        \includegraphics[scale=0.15]{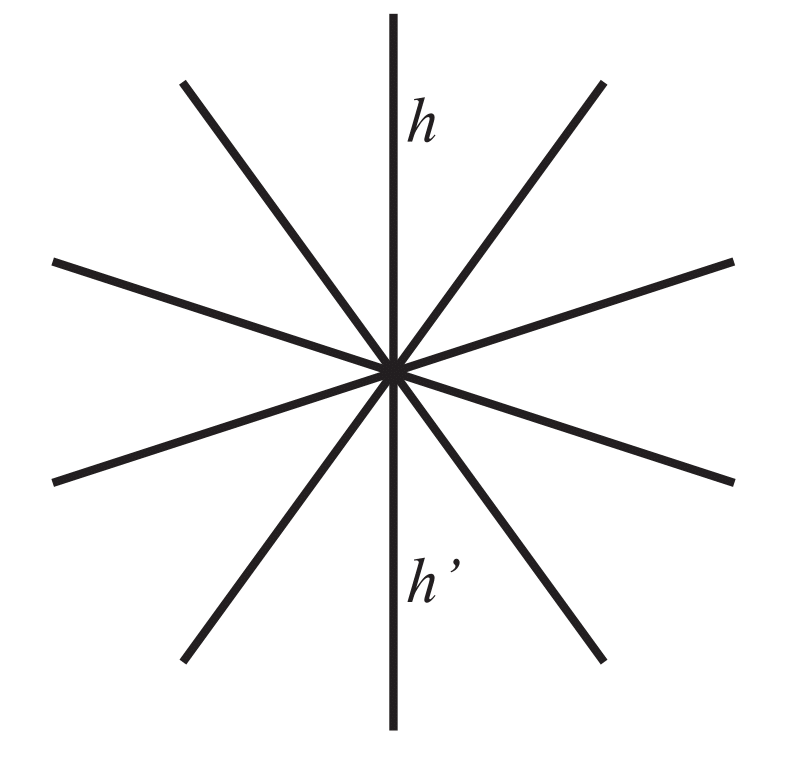}
        \caption{The vertex $v$ and its incident half-edges.  }
        \label{subfig:comp_smooths_1}
    \end{subfigure}%
    \begin{subfigure}[b]{0.5\columnwidth}
        \centering
        \includegraphics[scale=0.15]{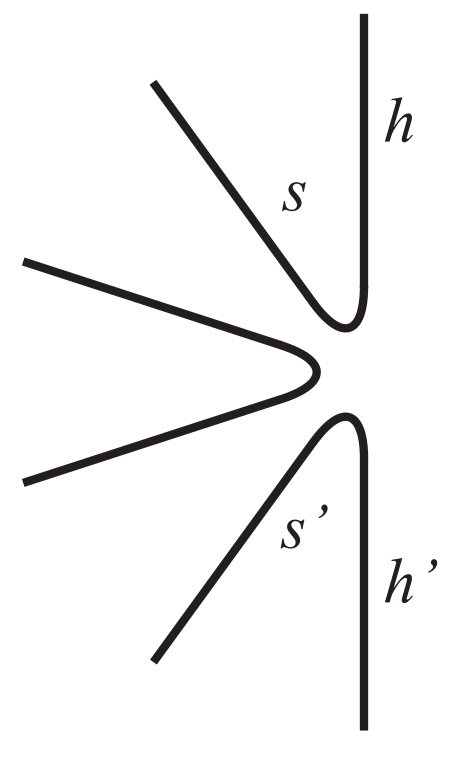}
        \caption{Smoothing $v_1$ according to $t_1$.}
        \label{subfig:comp_smooths_2}
    \end{subfigure}
    
    \begin{subfigure}[b]{0.5\columnwidth}
        \centering
        \includegraphics[scale=0.15]{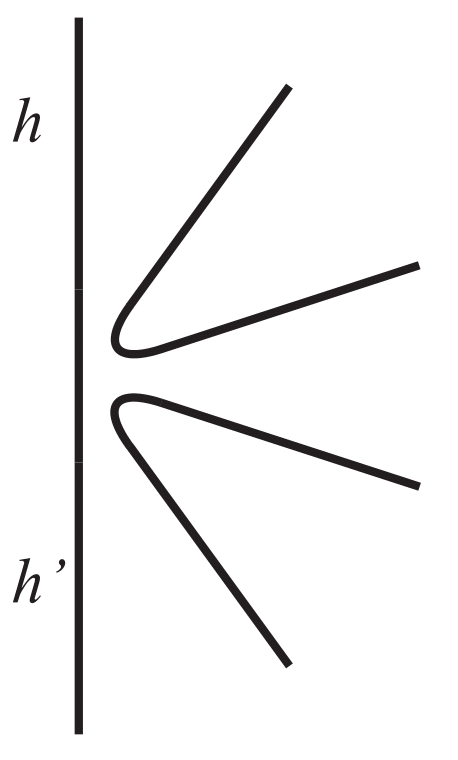}
        \caption{Smoothing $v_2$ according to $t_2$.}
        \label{subfig:comp_smooths_3}
    \end{subfigure}%
    \begin{subfigure}[b]{0.5\columnwidth}
        \centering
        \includegraphics[scale=0.15]{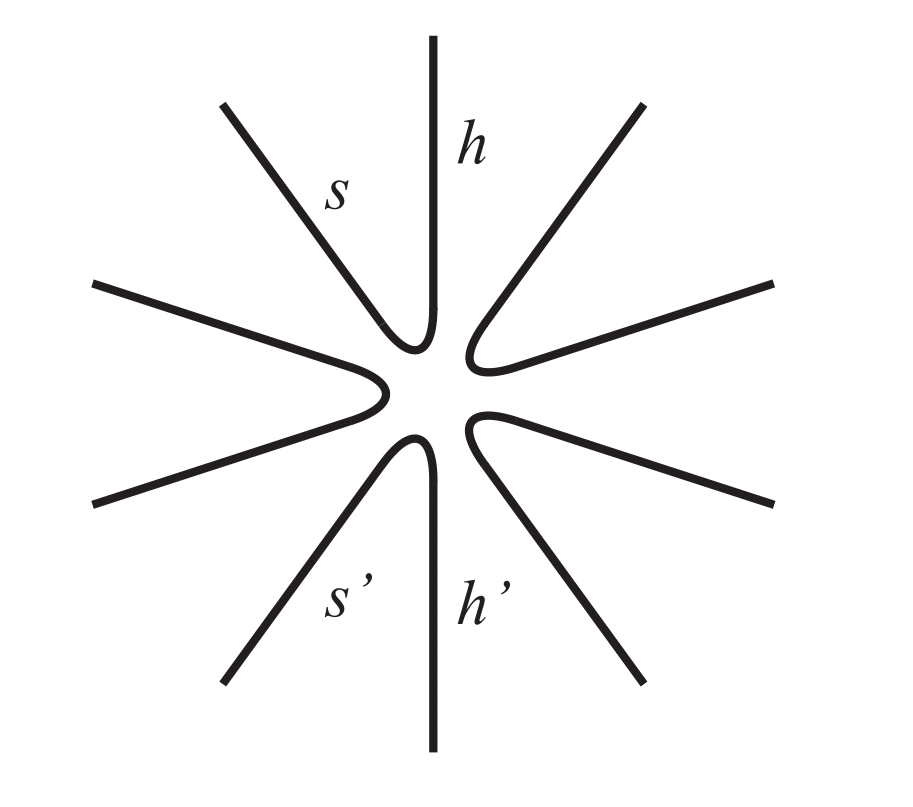}
        \caption{The transition $t_1 \# t_2$.}
        \label{subfig:comp_smooths_4}
    \end{subfigure}
    
    \caption{Forming the transition $t_1 \# t_2$ at a vertex $v$ of a composite grid.}
    \label{fig:comp_smooths}
\end{figure}

We can now define connected sums of transition systems for any number of torus graphs.

\begin{definition} Let $\{G_i\}_{i=1}^n$ be a collection of torus graphs with gluing set $\Gamma = \{ (f_{i,1}, f_{i,2},\pi_i) \}_{i=1}^{n-1}$. A set $\mathbb{T} = \{ T_i: T_i \text{ is a smooth transition of } G_i\}$ is \emph{compatible with $\Gamma$} if for each $1 \leq i < n$ the triple $(f_{i,1}, f_{i,2}, \pi_i)$ exhibits compatibility of $T_i$ and $T_{i+1}$. The \emph{connected sum} $\# \mathbb{T}$ is the smooth transition system constructed as follows. Any vertex of $G$ that is not formed by gluing vertices of two gluing faces inherits the smooth transition from the appropriate transition system $T_i$. Now suppose $v$ is a vertex formed by gluing $f_{i,1}$ to $f_{i,2}$. The transition at $v$ is taken to be the connected sum of the transitions at the preimages of $v$ in $T_i$ and $T_{i+1}$.
\end{definition}

Any smooth transition system of a composite grid can be written as the connected sum of smooth transition systems of the component grids.

\begin{theorem} \label{thm:connsum}
Let $G$ be a $\{G_i\}$-composite by $\{(f_{i,1},f_{i,2}, \pi_i)\}$. Let $T$ be a smooth transition system on $G$. Then there exists a set of smooth transition systems $\mathbb{T}$ of the component grids such that $T = \#\mathbb{T}$.
\end{theorem}
\begin{proof} Let $v$ be a vertex of $G$ formed by gluing a vertex of $G_i$ with a vertex of $G_{i+1}$. Let $t$ be the transition at $v$, and let $h$ and $h'$ be the half-edges at $v$ formed by gluing $f_{i,1}$ to $f_{i,2}$. By a parity argument, both of $h$ and $h'$ are paired by $t$ with half-edges in $G_i$, or both are paired by $t$ with half-edges in $G_{i+1}$. Thus we can reverse the process of Figure \ref{fig:comp_smooths} to obtain smooth transitions $t_i$ and $t_{i+1}$ at $v$ such that $t = t_i \# t_{i+1}$. Doing this at each vertex formed by gluing, and leaving transitions at unglued vertices unchanged, we obtain a list $\mathbb{T}$ of smooth transition systems of the component grids such that $T = \# \mathbb{T}$.
\end{proof}

We now restrict to composites of rectangular and triangular grids.

\begin{theorem} \label{thm:comptri}
   Suppose $G$ is a composite of triangular torus grids $\{G_i\}$ by.$\{(f_{i,1},f_{i,2}, \pi_i)\}$. Then $G$ contains an A-trail if and only if each component grid contains an A-trail, and any A-trail $G$ contains is necessarily unknotted.
\end{theorem}

\begin{proof} Suppose $G$ is a composite of triangular torus grids $\{G_i\}$ by $\{(f_{i,1},f_{i,2}, \pi_i)\}$. In Figure \ref{fig:S_tri_join} we provide the only pair of compatible transitions on faces of a triangular torus grid as well as their connected sum. Now suppose that each of these grids $G_i$ contains an A-trail. Using the construction given in Theorem \ref{thm:oddv}, we can choose smooth transition systems $T_i$ and $T_{i+1}$ such that $(f_{i,1},f_{i,2},\pi)$ exhibit compatibility of $T_i$ and $T_{i+1}$. Observe that in the connected sum of these transitions, the link $\mathcal{L}(\# \mathbb{T})$ crosses the cycle formed by gluing $f_{i,1}$ to $f_{i,2}$ exactly twice. Thus, $\mathcal{L}(\# \mathbb{T}) = \mathcal{L}(T_1) \# \cdots \# \mathcal{L}(T_n)$. Since the connected sum of unknots is the unknot, $\# \mathbb{T}$ is unknotted. Moreover, since each component transition system $T_i$ is an A-trail, we see that $\# \mathbb{T}$ contains every edge of $G$ and hence is an A-trail of $G$.

Now suppose $G$ contains an A-trail $T$. By Theorem \ref{thm:connsum}, $T = \#\mathbb{T}$ for some set $\mathbb{T} = \{T_i\}$ of smooth transition systems of the component grids. Consider some component transition system $T_i$. If $T_i$ is not an A-trail, then it has at least two distinct components $C_1$ and $C_2$. Note that when we take the connected sum of transition systems at a triangular face, exactly one component of one of the grids is connected to exactly one component of the other grid. Thus, since $\#\mathbb{T}$ crosses a pair of glued faces exactly twice, $C_1$ and $C_2$ are contained in different components of $\#\mathbb{T}$, a contradiction. Thus, each $T_i$ is an A-trail.
\end{proof}

\begin{figure}[ht!]
    \centering
    \begin{subfigure}{\columnwidth}
        \centering
        \includegraphics[scale=.3]{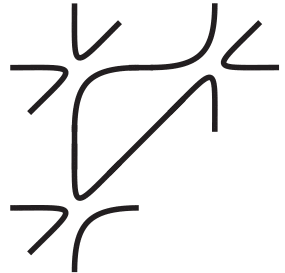} \qquad
        \includegraphics[scale=.3]{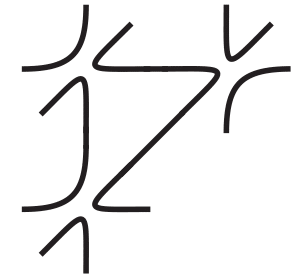}
        \caption{The only pair of compatible transitions.}
        \label{subfig:S_tri}
    \end{subfigure}
    
    \begin{subfigure}{\columnwidth}
        \centering
        \includegraphics[scale=.3]{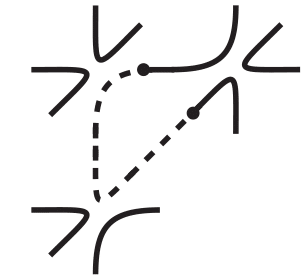} \qquad
        \includegraphics[scale=.3]{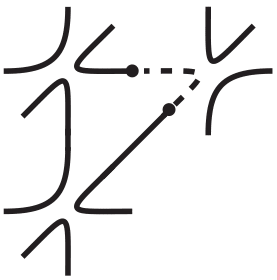}
        \caption{Their connected sum.}
        \label{subfig:S_tri_join}
    \end{subfigure}
    \caption{Connected sums of transitions in composite triangular grids.}
    \label{fig:S_tri_join}
\end{figure}

Theorem \ref{thm:comptri} does not duplicate for rectangular grids, as there are more pairs of compatible faces between these grids, some of which cross the glued face-boundaries more than twice (see Figure \ref{fig:S_ABC}). Indeed, in Figure \ref{fig:fig_8} we give an A-trail on the double torus, a figure-8 knot in fact, whose component smooth transition systems are two-component links. However, we can still construct infinitely many composite rectangular grids of any genus having unknotted A-trails.

\begin{figure}[ht!]
\centering
\includegraphics[scale=0.4]{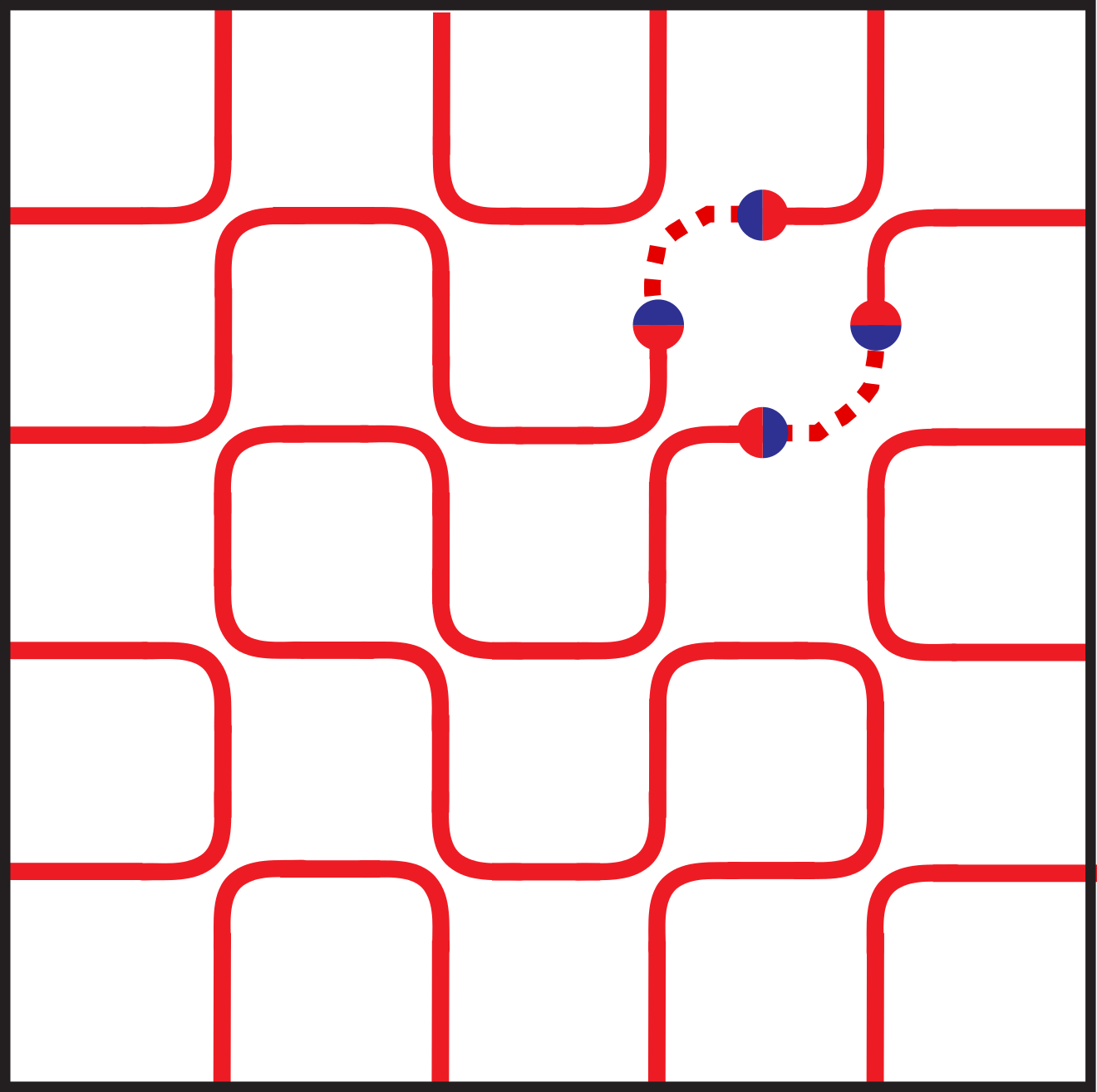} \qquad \qquad
    \includegraphics[scale=0.4]{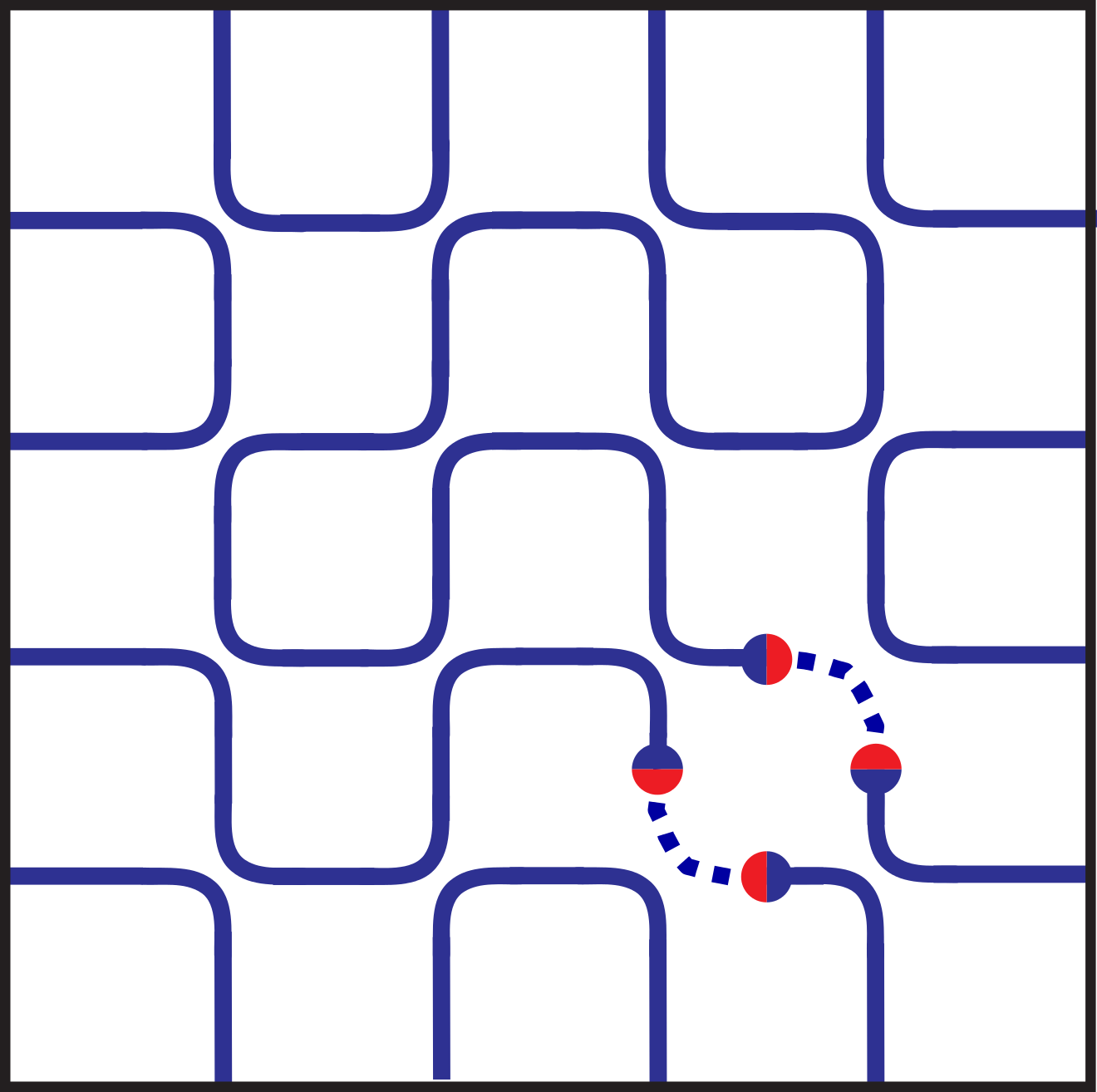}
    \caption{An A-trail on a composite grid of $\{R_{4,4}, R_{4,4}\}$. The A-trail forms a figure 8-knot, the component smooth transition systems each form a $2$-component link.}
    \label{fig:fig_8}
\end{figure}
\begin{theorem}For any $n$, there exist infinitely many composite rectangular grids of genus $n$ containing unknotted A-trails.
\end{theorem}

\begin{proof} In Figure \ref{fig:S_ABC} we give a list of all pairs of compatible transitions at rectangular faces and their corresponding connected sums. Using the unknotted A-trails of Theorem \ref{thm:unknot_rect}, we can find a set $\mathbb{T}$ of transition systems on rectangular grids $\{G_i\}$ of arbitrary dimensions inducing an unknotted A-trail. Moreover, we can choose these transition systems such that there exist faces $f_{i,1}$ and $f_{i,2}$ in $G_i$ and $G_{i+1}$ and isomorphisms $\pi: \partial f_{i,1} \to f_{i,2}$ such that $T_i$ and $T_{i+1}$ form a pair of type $S_A$ from Figure \ref{fig:S_ABC} with respect to $(f_{i,1},f_{i,2},\pi_i)$. Observe that in the connected sum of the transition set $S_A$, the link $\mathcal{L}(\# \mathbb{T})$ crosses the cycle formed by gluing $f_{i,1}$ to $f_{i,2}$ exactly twice. Thus, $\mathcal{L}(\# \mathbb{T}) = \mathcal{L}(T_1) \# \cdots \# \mathcal{L}(T_n)$. Since the connected sum of unknots is the unknot, $\# \mathbb{T}$ is unknotted. Moreover, since each component transition system $T_i$ is an A-trail, we see that $\# \mathbb{T}$ contains every edge of $G$ and hence is an A-trail of $G$.
\end{proof}

\begin{figure}[ht!]
    \centering
    \begin{subfigure}[b]{0.5\columnwidth}
        \centering
        \includegraphics[scale=0.15]{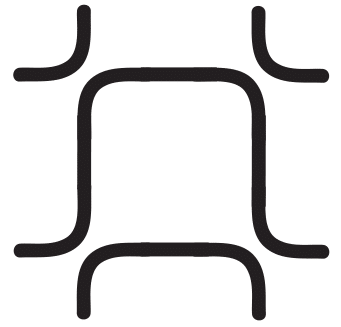}\qquad
        \includegraphics[scale=0.15]{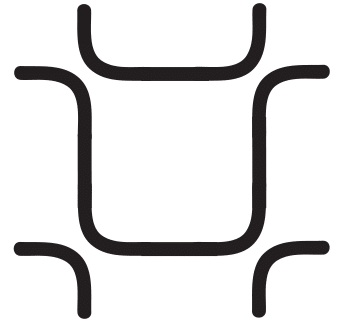}
        \caption{Transition set $S_A$.}
        \label{subfig:S_A}
    \end{subfigure}%
    \begin{subfigure}[b]{0.5\columnwidth}
        \centering
        \includegraphics[scale=0.15]{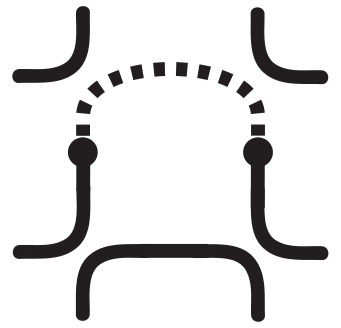}\qquad
        \includegraphics[scale=0.15]{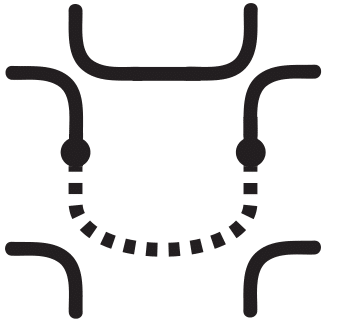}
        \caption{The connected sum of $S_A$.}
        \label{subfig:S_A_join}
    \end{subfigure}
    
    \begin{subfigure}[b]{0.5\columnwidth}
        \centering
        \includegraphics[scale=0.15]{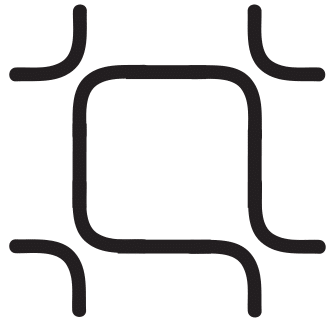}\qquad
        \includegraphics[scale=0.15]{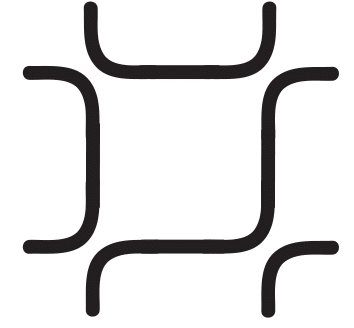}
        \caption{Transition set $S_B$.}
        \label{subfig:S_B}
    \end{subfigure}%
    \begin{subfigure}[b]{0.5\columnwidth}
        \centering
        \includegraphics[scale=0.15]{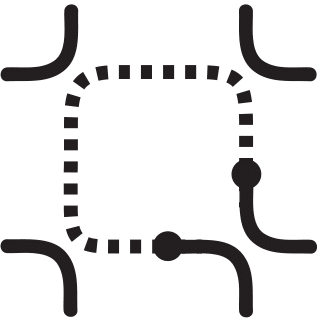}\qquad
        \includegraphics[scale=0.15]{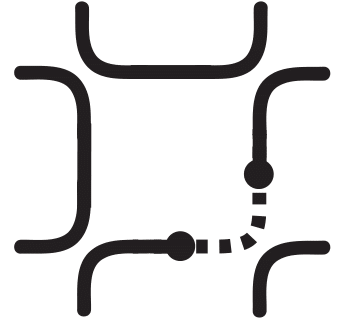}
        \caption{The connected sum of $S_B$.}
        \label{subfig:S_B_join}
    \end{subfigure}
    
    \begin{subfigure}[b]{0.5\columnwidth}
        \centering
        \includegraphics[scale=0.15]{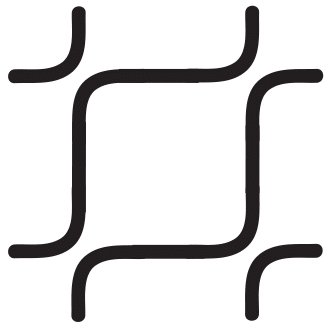}\qquad
        \includegraphics[scale=0.15]{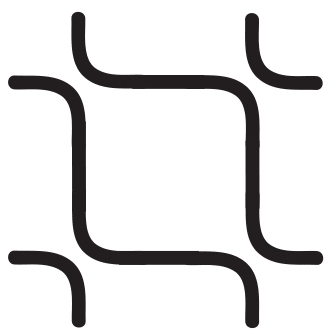}
        \caption{Transition set $S_C$.}
        \label{subfig:S_C}
    \end{subfigure}%
    \begin{subfigure}[b]{0.5\columnwidth}
        \centering
        \includegraphics[scale=0.15]{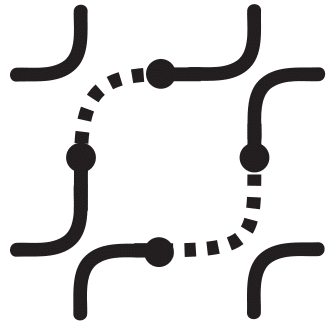}\qquad
        \includegraphics[scale=0.15]{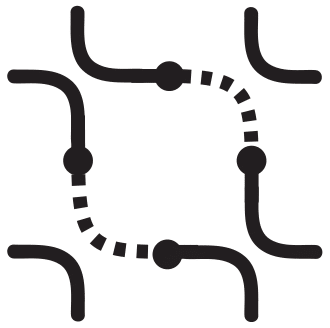}
        \caption{The connected sum of $S_C$.}
        \label{subfig:S_C_join}
    \end{subfigure}
    
    \caption{Pairs of compatible transitions and their connected sums.}
    \label{fig:S_ABC}
\end{figure}

\section{Conclusion} 

When the target of DNA origami self-assembly is modeled as a graph embedded on a surface, an optimal route for the circular (unknotted) scaffolding strand of DNA corresponds to an A-trail of the embedded graph. Thus, designing such DNA nanostructures using the origami method requires that the target embedding contain an unknotted A-trail. Note that the existence of an A-trail is a property only of the surface embedding, but the knottedness of the A-trail depends (in general) on the spatial embedding of the surface, complicating the analysis.

In terms of the motivating application, we have shown that if the target structure is checkerboard-colorable and toroidal, the embedding of the torus in space does not affect knottedness of any A-trails. Targets of DNA nanostructure self-assembly are often surface meshes, and we have characterized those non-SAH triangular torus grids as well as the rectangular torus grids $R_{i,j}$ containing unknotted A-trails. On surfaces other than the torus, we have constructed infinite families of triangular and rectangular meshes containing unknotted A-trails.

In addition to the practical application, we believe that the study of knotted A-trails (or linked smooth circuit decompositions) is of purely theoretical interest, as it provides a new perspective on what knotting means for embedded (Eulerian) graphs. With this in mind, we have shown that torus links can all be constructed from A-trails of rectangular grids, and that individual rectangular grids contain A-trails forming many different torus knots and links. 

We close by mentioning some open problems.

\begin{enumerate}
\item The following question was posed to the authors by Lou Kauffman: Let $K$ be a knot (or perhaps link) embedded on the surface $\Sigma$. Is there a rectangular grid $R$ on $\Sigma$ with an A-trail isotopic to $K$? What about links? We have shown that the answer to this question is \emph{yes} for the torus in Theorem \ref{thm:alltoruslinks}. Connected sums of torus knots can be constructed using connected sums of transition systems, and we have shown how to obtain on the double torus a Figure 8 knot, which is not a connected sum of torus knots, in Figure \ref{fig:fig_8}. However, the question is open in general.
\item We have shown that embedded graphs exist for which every A-trail is knotted. However, this example is in some sense trivial: viewed in $\mathbb{R}^3$, we have simply tied a knot in one loop of a two-loop graph. Are there nontrivial examples (and what, exactly, does ``trivial'' mean)?
\end{enumerate}

\section{Acknowledgements} The authors were supported by the National Science Foundation (NSF) under grant DMS-1001408. We would also like to thank Lou Kauffman and Joel Foisy for their insights at various stages of this research, and Ned Seeman for sharing his experience in DNA self-assembly for the problem formulation.

\bibliographystyle{abbrv}
\bibliography{Bibliography}

\end{document}